\documentclass[onefignum,onetabnum]{siamart190516}

\usepackage{amsmath,amssymb,nccmath}
\usepackage{amsfonts}
\usepackage{graphicx}
\usepackage{epstopdf}
\usepackage{enumerate} 
\usepackage{algorithmic}
\usepackage{comment}
\usepackage{xargs}
\usepackage{amsopn}
\usepackage{bbm}
\usepackage{relsize}
\usepackage{graphicx}
\usepackage{array,multirow,makecell}
\usepackage{cellspace}
\usepackage{booktabs}
\ifpdf
  \DeclareGraphicsExtensions{.eps,.pdf,.png,.jpg}
\else
  \DeclareGraphicsExtensions{.eps}
\fi

\newsiamremark{remark}{Remark}
\newsiamremark{hypothesis}{Hypothesis}
\crefname{hypothesis}{Hypothesis}{Hypotheses}
\newsiamthm{claim}{Claim}

\headers{Variance reduction for dependent sequences}{D. Belomestny, L. Iosipoi, E. Moulines, A. Naumov, and S. Samsonov}

\title{Variance reduction for dependent sequences with applications to Stochastic Gradient MCMC}

\author{Denis Belomestny\thanks{Duisburg-Essen University, Essen, Germany,  and HSE University, Moscow, Russia. Email:
  \email{denis.belomestny@uni-due.de}.}
\and Leonid Iosipoi\thanks{HSE University, Moscow, Russia. Email: 
  \email{liosipoi@hse.ru}.}
\and Eric Moulines\thanks{Ecole Polytechnique, Paris, France, and HSE University, Moscow, Russia. Email:
  \email{eric.moulines@polytechnique.edu}.} 
\and Alexey Naumov\thanks{HSE University, Moscow, Russia.  Email: 
  \email{anaumov@hse.ru}.}
\and Sergey Samsonov\thanks{HSE University, Moscow, Russia.  Email:   
  \email{svsamsonov@hse.ru}.}
}

\usepackage{amsmath}      
\usepackage{xargs}
\usepackage{hyperref}
\usepackage{nameref}
\usepackage{setspace}
\usepackage{mdframed}
\newmdenv[
    rightline=true,
    bottomline=false,
    topline=false,
    leftline=false,
    rightmargin = 0pt,
    innertopmargin=0pt,
    innerbottommargin=0pt,
    innerrightmargin=0pt,
    innerleftmargin=0pt,
    linewidth=0.75pt]{rline}
\newmdenv[
    rightline=true,
    bottomline=false,
    topline=false,
    leftline=false,
    rightmargin = 0pt,
    innertopmargin=0pt,
    innerbottommargin=0pt,
    innerrightmargin=0pt,
    innerleftmargin=0pt,
    linewidth=0.75pt]{rline2}

\makeatletter
\let\orgdescriptionlabel\descriptionlabel
\renewcommand*{\descriptionlabel}[1]{%
  \let\orglabel\label
  \let\label\@gobble
  \phantomsection
  \edef\@currentlabel{#1}%
  \let\label\orglabel
  \orgdescriptionlabel{#1}%
}
\makeatother

\renewcommand{\H}{\mathcal{H}}
\newcommand{\G}{\mathcal{G}}
\newcommand{\X}{\mathsf{X}}

\newcommand{\T}{\top}

\newcommand{\CM}{{R}}
\newcommand{\CH}{{D}}

\newcommand{\C}{{C}}
\newcommand{\Lnorm}{{L}}
\newcommand{\B}{{B}}

\newcommand{\ULA}{(\mathsf{ULA})}
\newcommand{\SGLD}{(\mathsf{SGLD})}

\DeclareMathOperator{\divergence}{div}

\newcommand{\rset}{\mathbb{R}} 
\newcommand{\nset}{\mathbb{N}}
\newcommand{\zset}{\mathbb{Z}}
\newcommand{\E}{\mathsf{E}} 
\newcommand{\Pb}{\mathsf{P}} 
\renewcommand{\P}{\mathsf{P}} 
\newcommand{\Vn}{V_{n}} 
\newcommand{\bVn}{\overline V_{n}} 
\newcommand{\eps}{\varepsilon} 
\newcommand{\sh}{h^*} 
\newcommand{\tih}{\tilde{h}}
\newcommandx{\hh}[1][1=n]{{\widehat{h}_{#1}}} 
\newcommand{\covcoeff}{\rho^{(h)}} 

\newcommand{\ntest}{N}
\newcommand{\SSize}{K}
\newcommand{\quant}{\texttt{q}}
\newcommand{\ecovcoeff}{\rho_n^{(h)}}  
\DeclareMathOperator*{\argmin}{arg\,min}

\def\Lip{\operatorname{Lip}}

\def\Id{\operatorname{I}}

\def\PCov{\mathsf{cov}}
\newcommandx\measureset[3][1=\mathrm{s},3=]{\mathbb{M}^{#3}_{#1}(\mathcal{#2})}

\newcommandx{\VnormFunc}[3][1=]{\ensuremath{\|#2\|_{{#3}}^{#1}}}

\newcommand{\lipnorm}[1]{\ensuremath{\left\Vert#1\right\Vert_{\mathsf{Lip}}}}

\newcommandx{\dens}[2][1=\theta]{\operatorname{p}_{{#1}}(#2)}

\def\Id{\mathrm{I}}
\def\rmd{\mathrm{d}}
\def\rme{\mathrm{e}}

\def\ltwo{L^2}
\def\mcf{\mathcal{F}}

\newcommand{\pscal}[2]{\langle #1,#2 \rangle}

\usepackage{todonotes}

\def\iid{i.i.d.}

\def\eqsp{\,}
\newcommandx\PVar[1][1=]{\operatorname{Var}_{{#1}}}
\def\lipU{{L}_U}
\def\tlipU{\widetilde{{L}}_U}
\def\mU{{m}_U}

\def\bx{\mathbf{x}}
\def\bX{\mathbf{X}}

\def\bI{\mathbf{I}}

\def\Xset{\mathsf{X}}
\def\sigXset{\mathcal{X}}

\def\Sset{\mathsf{S}}
\def\metric{\mathsf{d}}

\def\wascoef{\mathsf{\Delta}}
\def\Xkernel{\overline{P}}
\newcommandx{\PULA}[1][1=\gamma]{\ensuremath{P^{\tiny{\ULA}}_{#1}}}

\newcommand{\1}{\ensuremath{\mathbbm{1}}}
\newcommandx{\indi}[2][1=]{\1^{#1}_{#2}}
\newcommand{\indiacc}[1]{\1_{\{#1\}}}

\newcommand{\startdist}{\xi}
\newcommand{\KL}{\operatorname{KL}}
\newcommand{\batchsize}{{M}}

\def\FP{\operatorname{FP}}
\def\SAGA{\operatorname{SAGA}}

\newcommandx\spacexd[3][1=\metric,3=]{\mathbb{S}_{#3}(\mathsf{#2},#1)}

\ifpdf
\hypersetup{
  pdftitle={Variance reduction for dependent sequences   with applications to  Stochastic Gradient MCMC},
  pdfauthor={D. Belomestny, L. Iosipoi, E. Moulines, A. Naumov, and S. Samsonov}
}
\fi

\begin{document}

\maketitle

\begin{abstract}
In this paper we propose a novel and practical  variance reduction approach for  additive functionals of  dependent sequences. Our approach combines the use of control variates  with the minimisation of an empirical variance estimate. We analyse finite sample properties of the proposed method and derive finite-time bounds of the excess asymptotic variance to zero. We apply our methodology to Stochastic Gradient MCMC (SGMCMC) methods for Bayesian inference on large data sets and combine it with existing variance reduction methods for SGMCMC.  We present  empirical results carried out on a number of  benchmark examples showing that our variance reduction method  achieves significant improvement as compared  to state-of-the-art methods at the expense of a moderate increase of computational overhead.
\end{abstract}



\section{Introduction}
Variance reduction aims at reducing the stochastic error of a Monte Carlo estimate; see \cite{christian1999monte}, \cite{rubinstein2016simulation}, \cite{GobetBook}, and \cite{glasserman2013monte} for a an introduction to this field. Recently one witnessed a revival of interest in  variance reduction techniques  for dependent sequences with applications to Bayesian inference and reinforcement learning among others; see, for instance,
\cite{mira2013zero}, \cite{johnson:zhang:2013}, \cite{defazio2014saga},
\cite{chatterji2018theory}, \cite{baker2019control}, and references therein.
\par
Suppose that we wish to compute the  integral of an arbitrary function $f: \Xset \mapsto \rset$ with respect to a probability measure $\pi$ on a general state-space $(\Xset,\sigXset)$, that is, $\pi(f)= \int_{\Xset} f(x) \pi(\rmd x)$.  If
sampling \iid\ from $\pi$ is an option, a natural estimator for \(\pi(f)\) is the sample mean
\[
    \pi_{\ntest}(f):= {\ntest}^{-1} \sum\nolimits_{k=0}^{{\ntest}-1} f(X_k) \eqsp, \quad \ntest\in\nset,
 \]
where $(X_k)_{k=0}^{\ntest-1}$ is an i.i.d. sample from $\pi$.
Using the central limit theorem, one can construct an asymptotically valid confidence interval for the value $\pi(f)$ of the form
$ \pi_{\ntest}(f)\pm \quant \, {\ntest}^{-1/2} (\PVar[\pi](f))^{1/2}$,
where $\quant$ is a quantile of a normal distribution, and $\PVar[\pi](f)= \int_\Xset \{ f(x) - \pi(f) \}^2 \pi(\rmd x)$.
A general way to reduce the variance \(\PVar[\pi](f)\) is  to  select another function $g$ in a set $\G$ such that $\pi(g) = 0$ and $\PVar[\pi](f-g) \ll \PVar[\pi](f)$. Such a function $g$ is called a \emph{control variate} (CV).
A natural approach to learn $g \in \G$ is to minimize the empirical variance  
\begin{equation}\label{eq:evm}
    D_n(f-g)=(n-1)^{-1} \sum\nolimits_{k=0}^{n-1} \bigl(f(X_k)-g(X_k)-\pi_n(f-g)\bigr)^2.
\end{equation}
constructed using a new independent learning sample $(X_k)_{k=0}^{n-1}.$
This leads to the Empirical Variance Minimisation (EVM) method recently studied in \cite{belomestny2017variance} and \cite{biz2018}.
In many problems of interest, drawing an \iid\ sample  from $\pi$ is not an option, yet it is possible to obtain  a non-stationary dependent sequence $(X_k)_{k=0}^\infty$  whose marginal distribution converges to  \(\pi\). This situation is typical in Bayesian statistics, where $\pi$ represents a posterior distribution and $(X_k)_{k=0}^\infty$ is sampled using Markov chain Monte Carlo (MCMC) methods.
Under appropriate conditions, the central limit theorem also holds and therefore,
it is possible to construct  the asymptotic confidence interval for \(\pi(f)\) of the form
\begin{equation}
\label{eq:conf_int}
    \Biggl[\pi_{\ntest}(f) - \quant\, \sqrt{\frac{V_\infty(f)}{\ntest}},\pi_{\ntest}(f) + \quant\, \sqrt{\frac{V_\infty(f)}{\ntest}}\,\Biggr],
\end{equation}
where $V_\infty(f)$ is the asymptotic variance defined as
\begin{equation}
\label{V infinity}
V_\infty(f):=\lim_{\ntest \to \infty} {\ntest} \cdot \E \Bigl[\bigl(\pi_\ntest(f) - \pi(f) \bigr)^2\Bigr].
\end{equation}
A sensible approach is to select a control variate $g\in\G$ by minimizing an estimate for the asymptotic variance $V_\infty(f-g)$.
When  the spectral estimate of $V_\infty(f-g)$ is used, this leads to the Empirical
Spectral Variance Minimization (ESVM); see \cite{bimns2020}.
\par
In this paper, a special attention is paid to the case when $\Xset= \rset^d$ and $\pi$ admits a smooth and
everywhere positive density (also denoted by $\pi$) w.r.t to the Lebesgue measure, such that the gradient 
$\nabla U := - \nabla \log \pi$  can be evaluated.  We study below sampling methods derived from the discretization of the overdamped Langevin Dynamics (LD). It is defined by the following Stochastic Differential Equation:
\begin{align}\label{eq:contlangevindynamics}
\rmd Y_t = - \nabla U(Y_t)\, \rmd t + \sqrt{2} \rmd W_t \,,
\end{align}
where $(W_t)_{t \geq 0}$ is the standard Brownian motion. Note that
\(\nabla U \)
does not depend on the normalizing constant of $\pi$ which is typically  unknown in Bayesian inference. Under some technical conditions, the distribution of  $Y_t$ converges to $\pi$ as \(t\to \infty\), see \cite{roberts:tweedie:1996}.
The gradient-based MCMC algorithms are based on a time-discretized version of \eqref{eq:contlangevindynamics}.
In the Bayesian setting, a computational bottleneck of these algorithms is that the complexity of the  gradient \(\nabla U\) evaluation  scales proportionally to the number of observations (sample size) $\SSize$ which can be very time consuming in the ``big data" limit. To alleviate this problem, \cite{welling:the:2011} proposed to replace the "full" gradient $\nabla U$ by a stochastic gradient estimate based on sums over random \emph{minibatches}. This algorithm, Stochastic Gradient Langevin Dynamics (SGLD), has emerged as a key MCMC algorithm in Bayesian inference for large scale datasets.  The analysis of SGLD and its finite sample performance has attracted a wealth of contributions; see, for example, \cite{ma2015complete}, \cite{teh2016consistency}, \cite{nagapetyan2017true}, \cite{dalalyan:karagulyan:2017}, and the references therein.
These works show that the use of stochastic gradient comes at a price: while the resulting estimate of the gradient is still unbiased, its variance might annihilate the computational advantages of SGLD \cite{dalalyan:karagulyan:2017}. Several proposals have been made to reduce the variance of the stochastic gradient estimate of the ``full" gradient, inspired by several methods, proposed for incremental stochastic optimization; see \cite{leroux:schmidt:bach:2012}, \cite{johnson:zhang:2013}, and \cite{defazio2014saga}.  \cite{dubey2016variance} has investigated the properties of the Stochastic Average Gradient (SAGA) and Stochastic Variance Reduced Gradient (SVRG) estimators  for Langevin dynamics.
These results have been later completed and sharpened by \cite{dalalyan:karagulyan:2017}, \cite{chatterji2018theory}, \cite{brosse2018promises}. Other variance reduction approaches 
include various subsampling schemes and constructing alternative estimates for the gradient (see, for instance,  \cite{baker2019control} and \cite{zou2018subsampled}). 
\par
The paper is organized as follows.
In \Cref{sec:esvm}, we analyze the ESVM approach for general dependent sequences.  In particular, the ESVM method is described in \Cref{sec:esvm-method}. In \Cref{sec:theoretical-analysis}, we study the theoretical properties of the ESVM method for asymptotically stationary  dependent sequences. Here we provide a bound for the excess risk $V_\infty(f-\widehat{g}_n) - \inf\nolimits_{g \in \mathcal{G}} V_\infty(f-g)$,
where a control variate $\widehat{g}_n\in\G$ is chosen by minimization of the spectral variance $\Vn$ based on $(X_k)_{k=0}^{n-1}$, that is, $\widehat{g}_n \in \argmin \Vn (f-g)$. The precise definition of $\Vn$ will be given in \Cref{sec:esvm-method}.
In \Cref{sec:applications}, we apply these results to Markov chains which are uniformly geometrically ergodic in Wasserstein distance.  While \Cref{sec:langevin_ula} is devoted to the (undajusted) Langevin Dynamics, in \Cref{sec:esvm-stochgrad} we use the ESVM approach for variance reduction in  
SGLD-type algorithms. We show that in both cases,  
the excess variance can be bounded, with high probability and up to logarithmic factors, as
\begin{equation*}
    V_\infty(f-\widehat{g}_n) - \inf\nolimits_{g \in \mathcal{G}} V_\infty(f-g) = O\bigl( n^{-1/2}\bigr).
\end{equation*}
This implies asymptotically valid confidence intervals (conditional on the sample used to learn 
\(\widehat{g}_n \))  of the form
\[
\pi_{\ntest}(f-\widehat{g}_n) \pm \quant\,  \sqrt{\frac{\inf\nolimits_{g \in \mathcal{G}} V_\infty(f-g)+C n^{-1/2}}{\ntest}}
\]
for some constant \(C>0.\) Note that these intervals can be much tighter than ones in \eqref{eq:conf_int}, provided
that $n$ is large and \(\inf\nolimits_{g \in \mathcal{G}} V_\infty(f-g)\) is small. The latter condition is satisfied if the class \(\G\) is rich enough. 
In \Cref{sec:numerics}, we illustrate performance of the proposed variance reduction method on various benchmark problems.
\paragraph{Notations}
Let $(\X, \metric)$ be a complete separable metric space.
Define the Lipschitz norm of a real-valued function $h$
by $\lipnorm{h}: = \sup_{x \neq y \in \Xset} \{|h(y) - h(x)|/\metric(x,y)\}. $ 
We denote by $\Lip_\metric(\Lnorm)$ and $\Lip_{b,\metric}(\Lnorm, \B)$ the class of Lipschitz  (resp. bounded Lipschitz) functions on \(\X\) with $\lipnorm{h} \le \Lnorm$ (resp. $\lipnorm{h} \le \Lnorm$ and $|h|_\infty \le \B$). Further, let $\measureset[1]{\Xset}$ be a set of probability measures on $\Xset$. We denote for $p \geq 1$, $\spacexd[\metric]{\Xset}[p] : = \{ \lambda \in \measureset[1]{\Xset}: \int_\Xset \metric^p(x,y) \lambda(\rmd y) < \infty \text{ for all } x \in \Xset  \}$. For $\lambda, \nu \in \measureset[1]{\Xset}$, we denote their coupling set by \(\Pi(\lambda, \nu)\), i.e. \(\xi \in \Pi(\lambda, \nu)\) is the measure on $\Xset \times \Xset$ satisfying for all $A \in \mathcal B(\Xset)$, $\xi(A, \Xset) = \lambda(A)$ and $\xi(\Xset, A) = \nu(A)$. For $p \geq 1$ and $\lambda, \nu \in \spacexd[\metric]{\Xset}[p]$, let \(W_p^\metric(\lambda, \nu) := \inf_{\Pi(\lambda, \nu)} \{\int_{\Xset \times \Xset} \metric^p(x,y) \, \xi(\rmd x, \rmd y)\}^{1/p}\) be the Wasserstein distance of order $p$ between $\lambda$ and $\nu$.
For $\lambda, \nu \in \measureset[1]{\Xset}$, let $\KL(\lambda|\nu)$ be the Kullback-Leibler divergence of \(\lambda\) with respect to \(\nu\), i.e.,
$\KL(\lambda|\nu)= \int {\log (\rmd \lambda/\rmd \nu)} \rmd \lambda$ if $\lambda \ll \nu$ and $\KL(\lambda|\nu)=\infty$ otherwise.
Finally, unless otherwise specified,
the symbol $\lesssim$ stands for an inequality up to an absolute constant not depending on parameters of the problem.

\section{Empirical Spectral Variance Minimization}
\label{sec:esvm}
\subsection{Method}
\label{sec:esvm-method} Let $(\Omega, \mathfrak F, (\mathfrak F_k)_{k \geq 0}, \Pb)$ be a filtered probability space and $(X_k)_{k=0}^\infty$ be a random process adapted to the filtration $(\mathfrak F_k)_{k \geq 0}$ and taking values in $\Xset$. Let $f: \Xset \to \rset$ be a function such that $\pi(f^2) < \infty $ and $\E[f^2(X_k)] < \infty$ for all $k \in \nset$.
Let also $\G$ be a set of control variates, that is, functions $g\in\G$ satisfying $\pi(g^2)< \infty$, $\pi(g)= 0$, and $\E[g^2(X_k)] < \infty$ for all $k \in \nset$.
Particular examples of classes $\G$ are given below in \Cref{sec:applications}.
Denote the class of functions $h = f - g$ for $g\in\G$ by $\H$, that is,
\[
	\H := \{ f-g:\, g\in\G\}.
\]
To shorten notation, we shall write $\tih= h - \pi(h)$ for $h\in\H$.
\par
We impose the following
covariance stationarity condition on $(X_k)_{k=0}^{\infty}$
to ensure that the asymptotic variance $V_\infty(h)$ from~\eqref{V infinity} is well-defined  for any $h \in \H$.
\begin{description}
\vspace{0.5em}
\item[(CS)\label{CS}]
For any $h \in \H$, there exists a symmetric, summable, and positive semidefinite sequence $(\rho^{(h)}(\ell))_{\ell \in \zset}$ satisfying 
\begin{fleqn}[\parindent]
\begin{align*}
    &1)\ \rho^{(h)}(0) = \PVar[\pi](h),\\
    &2)\ \text{for any $\ell\in\nset_0$ and a constant $\CM>0$ independent of $h$ and $\ell$},\\
    &\ \quad\sum\nolimits_{k \in \nset_0} \Bigl|\E\bigl[\tih(X_k) \tih(X_{k+\ell}) \bigr]  - \rho^{(h)}(\ell)\Bigr|  \leq \CM,\\
    &3)\ \lim_{\ell\to \infty} \sum\nolimits_{k \in \nset_0}  \Bigl|\E\bigl[ \tih(X_k) \tih(X_{k+\ell})\bigr]  - \rho^{(h)}(\ell)\Bigr|  = 0.
\end{align*}
\end{fleqn}
\end{description}
\begin{proposition}
\label{asymptoticvariance}
Assume that the condition~\ref{CS} holds.  Then, for all \(h\in\H\),
the asymptotic variance \(V_\infty(h)\) defined in \eqref{V infinity}
exists and can be represented as
\begin{equation}
    \label{V_infinity_sum}
    V_\infty(h)=\sum\nolimits_{\ell \in \zset}   \rho^{(h)}(\ell).
\end{equation}
\end{proposition}
\begin{proof}
See \Cref{sec:proofasymptoticvariance}.
\end{proof}
The spectral variance estimator $V_n(h)$ is based on truncation and weighting  of the sample autocovariance functions: \begin{equation}
\label{eq:sv}
   \Vn(h) :=  \sum\nolimits_{|\ell|<b_n} w_n(\ell) \ecovcoeff(\ell),
\end{equation}
where $w_n$ is the lag window, $b_n$ is the truncation point, and
$\ecovcoeff(\ell)$ is the sample autocovariance function given,
for $\ell\in\nset_0$, by
\begin{equation}
\label{eq:empirical-autorcovariance}
 \ecovcoeff(\ell) = \ecovcoeff(-\ell):= n^{-1} \sum\nolimits_{k=0}^{n-\ell-1} \bigl(h(X_k) - \pi_n(h)\bigr) \bigl(h(X_{k+\ell}) - \pi_n(h)\bigr) .
\end{equation}
Here the truncation point $b_n$ is an integer depending on $n$ and the lag window $w_n$ is a kernel of the form \(w_n(\ell)=w(\ell/b_n)\),
where \(w\) is a symmetric non-negative function supported on \([-1,1]\)
such that \(\sup_{y\in[0,1]}|w(y)|\leq 1\) and \(w(y)=1\) for \(y \in [-1/2,1/2]\). There are several other estimates for the asymptotic variance $V_{\infty}(h)$; see \cite{flegal:jones:2010} and the references therein.
The ESVM estimator is defined via
\begin{equation}\label{eq:estimate_wholeH}
	\widehat{h}_n \in \argmin\nolimits_{h\in\H}\Vn\bigl(h\bigr).
\end{equation}
The ESVM algorithm 
is summarized in \Cref{algorithm:esvm}.

\begin{algorithm}[!th]
\setstretch{1.35}
\caption{Empirical Spectral Variance Minimization (ESVM) method\label{algorithm:esvm}}
\begin{algorithmic}
   \STATE {\bfseries Input:}  Two independent sequences: $\mathbf{X}_n=(X_k)_{k=0}^{n-1}$
   and $\mathbf{X}'_{\ntest}=(X'_k)_{k=0}^{\ntest-1}$.
   \STATE \textbf{1.} Choose a class $\G$ of functions with $\pi(g)=0$ for all functions $g\in\G$.
   \STATE \textbf{2.} Find  $\widehat{g}_n\in\argmin_{g\in\G}\Vn(f-g)$, where $\Vn$ is computed as in \eqref{eq:sv}.
    \STATE{\bfseries Output:}
    \(\pi_{\ntest}\left(f - \widehat{g}_n\right) \) 
    computed based on  $\mathbf{X}'_{\ntest}$.
\end{algorithmic}
\end{algorithm}

\subsection{Theoretical analysis}
\label{sec:theoretical-analysis}
For our theoretical analysis, instead of looking for a function with smallest spectral variance in the whole class $\H$ we will perform optimization  over a finite approximation (net) of $\H$. It turns out that the both estimators have similar theoretical properties.
Fix some $\eps>0$. 
Assuming that the class $\H$ is totally bounded in $\ltwo(\pi)$, let $\H_{\eps}$ be {a minimal} $\eps$-net in 
$\ltwo(\pi)$-norm, that is, the smallest possible (finite) collection of functions $\H_{\eps} \subset \H$ with the property that
for any $h \in \H$ there exists $h_{\eps} \in \H_{\eps}$
such that the distance between $h$ an $h_{\eps}$ in $\ltwo(\pi)$-norm is less than or equal to $\eps$. The cardinality of $\H_{\eps}$ is called
the covering number and is denoted by $|\H_{\eps}|$.
Define
\begin{equation*}
	\widehat{h}_{n,\eps} \in \argmin\nolimits_{h \in \H_{\eps}} \Vn(h).
\end{equation*}
To obtain a quantitative bound for  the
asymptotic variance of $\widehat{h}_{n,\eps}$, we need to specify the decay rate of the sequence $(\rho^{(h)}(\ell))_{\ell\in\zset}$ from \ref{CS}. 
\begin{description}
\vspace{0.5em}
\item[(CD)\label{CD}]
There exist $\varsigma > 0$ and $\lambda \in [0,1)$ such that, for any $h\in\H$ and $ \ell \in \nset_0$,
\[
    \bigl|\rho^{(h)}(\ell)\bigr| \le \varsigma  \lambda^{\ell}.
\]
\end{description}
The following theorem provides a general bound on the excess of asymptotic variance.
\begin{theorem}\label{th:main}
Assume that the conditions \ref{CS} and \ref{CD} hold.
Assume additionally that for any $n\in \mathbb{N}$ there exists a  decreasing continuous function
$\alpha_n$ satisfying
    \[
    \sup\nolimits_{h\in\H}\P\Bigl(\bigr|\Vn(h) -\E[\Vn(h)] \bigl|>t\Bigr)
    \leq
    \alpha_n(t), \quad t>0.
    \]
Then, for any $\delta\in(0,1)$ and \(\eps>0\), it holds with probability at least $1-\delta$ that
\begin{multline*}
    V_\infty(\hh[n,\eps]) - \inf\nolimits_{h\in\H}V_\infty(h)
    \lesssim
    \alpha_n^{-1}\biggl(\frac{\delta}{2|\H_{\eps}|}\biggr)
    +\bigl(\sqrt{\CM}n^{-1/2}+\sqrt{\CH}\bigr)b_n \eps
    \\
	+\sqrt{\CM\CH}\,b_nn^{-1/2}
    +\bigl(\CM + \varsigma (1 - \lambda)^{-1}\bigr)b_nn^{-1}
    +\varsigma (1 - \lambda)^{-2}n^{-1}
    \\
    +\varsigma (1 - \lambda)^{-1}\lambda^{b_n/2},
\end{multline*}
where $\alpha_n^{-1}$ 
is an inverse function for $\alpha_n$ 
and $\CH = \sup\nolimits_{h \in \H}  \PVar[\pi](h)$.
\end{theorem}

\begin{proof}
See \Cref{sec:proofthmain}.
\end{proof}
\par
Under  some additional assumptions on the covering number of $\H$  and the  function $\alpha_n(t)$, a suitable choice of the size of $\eps$-net and the truncation point $b_n$, yields the following high-probability bound
\begin{equation*}
    V_\infty(\hh[n,\eps]) - \inf\nolimits_{h\in\H}V_\infty(h)
    \lesssim
    n^{-1/(2+\rho)}
    \quad \text{for some}\
    \rho>0,
\end{equation*}
where \(\lesssim\) stands for inequality up to a constant depending on \(\lambda\), \(\CM\), \(\CH\),
and \(\varsigma\). In the next section we  shall apply \Cref{th:main} to the analysis of the ESVM algorithm for dependent sequences in ULA and SGLD.
\section{Applications}\label{sec:applications}
In general, \Cref{th:main} can be applied to different types of dependent sequences satisfying conditions \ref{CS} and \ref{CD}.
In what follows, we let $(\Xset, \metric)$ be a complete separable metric space  (equipped with its Borel $\sigma$-algebra $\sigXset$) and consider $P$ to be a Markov kernel on $(\Xset,\sigXset)$. Let $\Omega= \Xset^\nset$ be the set of $\Xset$-valued sequences endowed with the $\sigma$-field $\mathfrak F= \sigXset^{\nset}$, $(X_k)_{k=0}^{\infty}$ be the coordinate process, and $\mathfrak F_k= \sigma(X_\ell, \ell \leq k)$ be the canonical filtration. For every probability measure $\xi$ on $(\Xset,\sigXset)$ there exists a unique probability $\Pb_{\xi}$ on $(\Xset^\nset,\sigXset^{\otimes \nset})$ such that the coordinate process $(X_k)_{k=0}^\infty$ is a Markov chain with Markov kernel $P$ and initial distribution $\xi$. We denote by $\E_{\xi}$ the associated expectation. We focus below on the case where $P$ is $W_p^{\metric}$-uniformly ergodic for $p=1$ or $p=2$.
\\
\begin{description}
\vspace{-0.5em}
\item[(WE)\label{WE}]\hspace{-5pt}-$p$ \
There exists $x_0 \in \Xset$ such that $\int_\Xset \metric(x_0,x) P(x_0,\rmd x) < \infty$ and a constant $\wascoef_p(P) \in [0,1)$ such that
\[
\sup_{(x,x^{\prime}) \in \Xset^{2},\, x \neq x^{\prime}}
\frac{W_p^\metric(\delta_{x}P, \delta_{x^{\prime}}P)}{\metric(x,x^{\prime})} = \wascoef_p(P) \eqsp.
\]
\end{description}
\cite[Theorem~20.3.4]{douc:moulines:priouret:2018} shows that if \ref{WE}-$p$ holds for some $p \geq 1,$ then $P$ admits a unique invariant probability measure which is denoted by $\pi$ below. Moreover, $\pi \in  \spacexd[\metric]{X}[p]$ and for any $\xi \in \spacexd[\metric]{X}[p],$ 
  \begin{equation}
    \label{eq:wasser:unif:thm}
    W_p^\metric(\xi P^n,\pi)\leq
     \wascoef^{n}_p(P)
    W_p^\metric(\xi,\pi) \eqsp, \quad n\in \nset.
  \end{equation}
If there is no risk of confusion, we denote for simplicity $\wascoef_p= \wascoef_p(P)$.
Let us start with  
a general result for Markov kernels satisfying \ref{WE}-$2$.
We show below that this assumption implies \ref{CS} and \ref{CD} when $\H$ is a subset of Lipschitz functions, and establish an exponential concentration inequality for $\Vn(h)$, $h\in\H$. As it was emphasized in \cite{marton1996} and \cite{Guillin2004}, powerful tools for exploring concentration properties of $W_2^\metric$-ergodic Markov kernels are the transportation cost-information inequalities.
\begin{definition*}
For $p \geq 1$, we say that $\mu\in \mathbb M_1(\Xset)$ satisfies $L^p$-transportation cost-information inequality 
with constant $\alpha>0$ if for any  $\nu \in \mathbb M_1(\Xset)$,
$W_p^{\metric}(\mu,\nu) \leq \sqrt{2\alpha \KL(\nu|\mu)}.$
We write briefly $\mu \in T_p(\alpha)$
for this relation.
\end{definition*}

\par
$L^p$-transportation cost-information inequalities are well-studied in the literature, see, for instance, \cite{GentilBakryLedoux} and references therein. The cases $p=1$ and $p=2$ are of particular interest. Relations between $T_1(\alpha)$ and concentration inequalities are covered in \cite{ledoux:2001} and \cite{BobkovGotze}. In particular, $T_1(\alpha)$ is known to be equivalent to Gaussian concentration for all Lipschitz functions, see \cite{BobkovGotze}. 
In turn $T_2(\alpha)$ is a stronger inequality than $T_1(\alpha)$.
It was first established for the standard Gaussian measure on $\rset^d$ by Talagrand in \cite{talagrand:1996}. 
Moreover, the celebrated result by Bakry-Emery \cite{bakry-emery:1985} implies that the measure $\pi(dx) = \rme^{-U(x)} \rmd x$ satisfies $T_2(\alpha)$ if $\nabla^2 U \geq \alpha^{-1} \Id$, see \cite[Chapter~9.6]{GentilBakryLedoux}. We are especially interested in $T_2(\alpha)$, since it is known to be stable under both independent and Markovian tensorisations, see \cite{otto_villani:2001} and \cite{Guillin2004}. 
\par
Our results on $W_2^{\metric}$-ergodic Markov kernels are summarized below.
\begin{proposition}
\label{prop:W2_concentration}
Let $\mathcal{H} \subseteq \Lip_{\metric}(\Lnorm)$ and
assume that \ref{WE}-2 holds.
Then, for any initial distribution $\xi \in \spacexd[\metric]{\Xset}[2]$,
\ref{CS} is satisfied with
\begin{equation}\label{eq:constatsW2_1}
    \covcoeff(\ell) = \E_\pi\bigl[\tih(X_0)\tih(X_{|\ell|})\bigr],
    \quad
    \CM = A_1 L^{2}(1 - \wascoef_2)^{-1} W_2(\startdist, \pi),
\end{equation}
where $A_1$ is a constant given in \eqref{eq:definition-C-1}, and \ref{CD} is satisfied with
\begin{equation}\label{eq:constatsW2_2}
    \varsigma=
     L\sqrt{\CH} \biggl[\int \{W_2^{\metric}(\delta_{x},\pi)\}^2\pi(\rmd x)\biggr]^{1/2},
     \quad
     \lambda = \wascoef_2,
     \quad
     \CH = \sup\nolimits_{h \in \H}\PVar[\pi](h).
\end{equation}
Moreover, if $P(x,\cdot) \in T_2(\alpha)$ for any $x \in \Xset$ and some $\alpha > 0$, then, for any initial distribution $\startdist \in T_2(\alpha)$, $n \in \nset$, and $t > 0$,
\begin{equation}
    \label{eq:conc_vn_w2}
	    \Pb_{\startdist}\bigl(\bigl|\Vn(h) - \E_{\startdist}[\Vn(h)]\bigr| \geq t\bigr) \leq 2\exp{\biggl(-\frac{(1-\wascoef_2)^2nt^2}{c\alpha L^2b_n^2\big(\CH + \CM n^{-1} + t\big)}\biggr)}\eqsp,
\end{equation}
where $c>0$ is an absolute constant.
\end{proposition}
\begin{proof}
See \Cref{sec:prop_W2}.
\end{proof}
\par
It is also possible to remove a quite restrictive assumption $P(x,\cdot) \in T_2(\alpha)$ and to relax \ref{WE}-$2$ to \ref{WE}-$1$, but in this case \ref{CS} and \ref{CD} can be verified only for $\H$ being a subset of bounded Lipschitz functions. As a price for such generalisation, the exponential concentration bound is replaced by a polynomial one.
\begin{proposition}
\label{prop:W1_concentration}
Let $\H \subset \Lip_{b,\metric}(\Lnorm, \B)$ and assume that \ref{WE}-1 holds. Then for any initial distribution $\xi \in \spacexd[\metric]{\Xset}[1]$, \ref{CS} is satisfied  with
\begin{equation}\label{eq:constatsW1_1}
    \covcoeff(\ell) = \E_\pi\bigl[\tih(X_0)\tih(X_{|\ell|})\bigr],
    \quad
    \CM = A_2 B(1-\wascoef_1^{1/2})^{-1},
\end{equation}
where $A_2$ is a constant given in \eqref{eq:A2},
and \ref{CD} is satisfied  with  
\begin{equation}\label{eq:constatsW1_2}
    \varsigma = 2 L\B \int W_{1}^{\metric}(\delta_{x},\pi)\pi(\rmd x),
    \quad
    \lambda = \wascoef_1,
    \quad
    \CH = \sup\nolimits_{h \in \H}\PVar[\pi](h).
\end{equation}
Moreover, for any  $p \in \nset$,
    \begin{equation}
    \label{eq:rosenthal}
    \Pb_{\startdist}\bigl(\bigl|\Vn(h) - \E_{\xi}[\Vn(h)]\bigr| \geq t \bigr) \leq
    \frac{\C_{\mathsf{R},1}^{p}B^{2p}b_n^{3p/2}p^{p}}{n^{p/2}t^p} + \frac{\C_{\mathsf{R},2}^pB^{2p}b_n^{2p}p^{2p}}{n^{p-1}t^p},
    \end{equation}
    where constants $\C_{\mathsf{R},1}$ and $\C_{\mathsf{R},2}$ are given in \eqref{eq:fin_bound}.
\end{proposition}
\begin{proof}
See \Cref{sec:prop_W1}.
\end{proof}

\subsection{Langevin dynamics}\label{sec:langevin_ula} In this case, $\Xset=\rset^d$ and we assume that $\pi$ has an everywhere positive density w.r.t the Lebesgue measure, that is, $\pi(\theta)= Z^{-1} \rme^{-U(\theta)}$, where $Z= \int \rme^{-U(\vartheta)} \rmd \vartheta$ is the normalization constant. Consider the first-order Euler-Maruyama discretization of
the Langevin Dynamics from \eqref{eq:contlangevindynamics},
\begin{equation}
 \label{eq:discretized}
 \theta_{k+1} =\theta_k - \gamma \nabla U(\theta_k) + \sqrt{2\gamma}\,\xi_{k+1},
\end{equation}
where $\gamma > 0$ is a step size and $(\xi_k)_{k=1}^\infty$ is an \iid\ sequence of the standard Gaussian $d$-dimensional random vectors. The idea of using \eqref{eq:discretized} to approximately sample from 
$\pi$ has been advocated by \cite{roberts:tweedie:1996} which coin the term Unadjusted Langevin Algorithm (ULA). Consider the following assumption on $U$.
\begin{description}
\item[(ULA)\label{ULA}]
The function \(U\) is continuously differentiable on \(\rset^d\) with  gradient \(\nabla U\)
satisfying the following two conditions.
\begin{enumerate}[1)]
\item Lipschitz gradient: there exists $\lipU > 0$ such that for all $\theta, \theta' \in \rset^d$ it holds that
$\|\nabla U(\theta) - \nabla U(\theta')\| \leq \lipU \|\theta-\theta'\|;$
\item  Strong convexity: there exists a constant $\mU > 0$, such that for all \( \theta,\theta' \in \rset^d\) it holds that
$U(\theta') \geq U(\theta) + \pscal{\nabla U(\theta)}{\theta'-\theta} + (\mU/2)\|\theta'-\theta\|^2 $.
\end{enumerate}
\vspace{0.5em}
\end{description}
The Unadjusted Langevin Algorithm has been widely studied under the above assumptions, see, for example, \cite{durmus:moulines:2018} and \cite{dalalyan2014}. As it is known from \cite{durmus:moulines:2018}, under \ref{ULA} the associated Markov kernel, denoted by $\PULA$, is $W_2^{\metric}$-uniformly ergodic. For completeness, we state below \cite[Proposition 3]{durmus:moulines:2018}.
\begin{proposition}
\label{prop:ULA_contractivity} Assume \ref{ULA} and set $\kappa = 2\mU \lipU / (\mU + \lipU)$. Then for any step size $\gamma \in (0,2/(\mU + \lipU))$, $\PULA$ satisfies \ref{WE}-2 with $\metric(\vartheta, \vartheta') = \|\vartheta - \vartheta'\|$ and $\wascoef_2 = \sqrt{1-\kappa \gamma}$. Moreover, $\PULA[\gamma]$ has a unique invariant measure $\pi_{\gamma}^{\ULA}$. 
\end{proposition}
It is shown in \cite[Corollary~7]{durmus:moulines:2018} that, for any step size $\gamma \in (0,2/(\mU + \lipU))$,
\begin{equation*}
W_{2}^{\metric}\bigl(\pi, \pi_{\gamma}^{\ULA}\bigr) \leq \sqrt{2} \kappa^{-1/2} \lipU \gamma^{1/2} \left\{\kappa^{-1}+\gamma\right\}^{1/2} \{2 d+d \lipU^{2} \gamma / \mU+d \lipU ^{2} \gamma^{2} / 6\}^{1/2} \eqsp.
\end{equation*}
We define the asymptotic variance as
\begin{equation*}
    V_\infty^{\ULA}(h) := \sum\nolimits_{\ell \in \zset} \E_{\pi_{\gamma}^{\ULA}}\Bigl[
    \bigl(h(X_0) - \pi_{\gamma}^{\ULA}(f)\bigr)
    \bigl(h(X_{|\ell|}) - \pi_{\gamma}^{\ULA}(f)\bigr)
    \Bigr].
\end{equation*}
\par
At each iteration of the algorithm,  $\nabla U$ is computed. Hence it is an appealing option to use this gradient to construct  Stein control variates (see, for instance, \cite{assaraf1999}, \cite{mira2013zero}, and \cite{oates2017control}), given by
\begin{equation}
\label{eq:stein}
g_\phi(\theta) = - \pscal{\phi(\theta)}{\nabla U(\theta)} + \divergence\bigl(\phi(\theta)\bigr),
\end{equation}
where $\phi: \Xset \to \rset^{d}$ is a continuously differentiable Lipschitz function, $\pscal{\cdot}{\cdot}$ is the standard scalar product in $\rset^d$, and $\divergence(\phi)$ is the divergence of $\phi$. Under rather mild conditions on $\pi$ and $\phi$, it follows from integration by parts that $\pi(g_{\phi})=0$ (see \cite[Propositions~1~and~2]{mira2013zero}).
Note that if $\phi(\theta) \equiv b$, $b \in \rset^d$, we get $g_{b}(\theta) = - \pscal{b}{\nabla U(\theta)}$. Then for a parametric class $\H = \{f - g_b :\,\|b\| \leq B\}$, assuming that $f \in \Lip_{\metric}(L_1)$ and that condition \ref{ULA} holds, we get $\H \subset \Lip_{\metric}(\max(L_1,B\lipU))$.
For other approaches to construct control variates we refer  reader to
\cite{henderson1997variance}, \cite{dellaportas2012control}, and \cite{brosse2018diffusion}.
The next result follows now from \Cref{th:main} and \Cref{prop:W2_concentration}.
\begin{theorem}
\label{th:main:ula}
Let $\H\subset \Lip_\metric(\Lnorm)$ and assume that
\ref{ULA} holds.
Assume additionally that  $\startdist \in T_2(\beta)$ for some $\beta > 0$.
Fix any $\gamma \in (0,2/(\mU + \lipU))$ and set $b_n= 2 \lceil\log(n)/\log(1/\wascoef_2) \rceil\eqsp$ with $\wascoef_2 = \sqrt{1 - \kappa \gamma}$ and $\kappa = 2\mU \lipU / (\mU + \lipU)$. Then, for any $\eps > 0$ and $\delta\in(0,1)$, with probability at least $1-\delta$,
\begin{multline*}
    V_\infty^{\ULA}(\hh[n,\eps]) - \inf\nolimits_{h\in\H}V_\infty^{\ULA}(h) \\
    \lesssim
    \C_1 \, \eps \log(n)
    +\C_2 \sqrt\frac{\log^2(n){\log(|\H_\eps|/\delta)}}{{n}}
    +\C_3\, \frac{\log^2(n)\log{(|\H_\eps|/\delta)}}{n},
\end{multline*}
where
\begin{align*}
&\C_1 = \frac{\sqrt{R} + \sqrt{D}}{\kappa \gamma},\,
\C_2 = \frac{L\sqrt{(\beta \vee \gamma)(\CH + \CM)}}{\kappa^2\gamma^{2}} + \frac{\sqrt{\CH \CM}}{\kappa \gamma},\, 
\C_3 = \frac{L^2(\beta \vee \gamma)}{\kappa^4\gamma^4} + \frac{\CM}{\kappa \gamma} + \frac{\varsigma}{\kappa^2 \gamma^2}
\end{align*}
with $\CM$, $\varsigma$ from \Cref{prop:W2_concentration}
and $\CH = \sup_{h\in\H}\PVar[\pi_{\gamma}^{\ULA}](h)$.
\end{theorem}
\begin{proof}
The Markov kernel associated to ULA can be written as $\PULA(\theta,\cdot) = \mathcal{N}(\theta- \gamma \nabla U(\theta), \, 2\gamma I_d)$. Hence, by~\cite[Theorem~9.2.1]{GentilBakryLedoux},  $\PULA(\theta,\cdot) \in \operatorname{T}_2(2\gamma)$ for any $\gamma > 0$. By \Cref{prop:ULA_contractivity}, \ref{WE} holds with $\wascoef_2 = \sqrt{1-\gamma\kappa}$. Hence \Cref{prop:W2_concentration} applies with $\alpha = 2(\beta \vee \gamma)$. Direct computation of the inverse function in the right-hand side of \eqref{eq:conc_vn_w2} leads to
\[
\alpha_n^{-1}\biggl(\frac{\delta}{2|\H_\eps|}\biggr) \leq \frac{4 b_n^2 L^2 (\beta \vee \gamma) \log{(4|\H_\eps|/\delta)}}{(1-\wascoef_2)^{2}n} + \frac{4 b_n L \sqrt{(\beta \vee \gamma)(\CH + \CM)\log{(4|\H_\eps|/\delta)}}}{(1-\wascoef_2)\sqrt{n}}.
\]
\end{proof}
\begin{corollary}
\label{cor:ula}
Under the assumptions of \Cref{th:main:ula}, the following holds.
\begin{enumerate}[1)]
\vspace{0.5em}
    \item if class \(\H\) is parametric, that is, \(|\H_\eps|\leq C_{\rho}\eps^{-\rho}\) for all \(\eps\in(0,1)\) and some constants \(C_\rho,\rho>0\), then it holds with probability at least \(1-1/n,\)
    \[
    V_\infty^{\ULA}(\hh[n,\eps]) - \inf\nolimits_{h\in\H}V_\infty^{\ULA}(h)\lesssim n^{-1/2}\log^{1/2}(n),
    \]
    \item if class \(\H\) is non-parametric, that is, \(|\H_\eps|\leq C_{\rho}\exp(\eps^{-\rho})\) for all \(\eps\in(0,1)\) and some constants \(C_{\rho},\rho>0\), then it holds with probability at least \(1-1/n,\)
    \[
    V_\infty^{\ULA}(\hh[n,\eps]) - \inf\nolimits_{h\in\H}V_\infty^{\ULA}(h)\lesssim n^{-1/(2+\rho)}.
    \]
\end{enumerate}
Here \(\lesssim\) stands for inequality up to a constant depending on \(\rho\) and other constants from \Cref{th:main:ula}. Moreover, if additionally the constant \(\pi_{\gamma}^{\ULA}(f) \) is in the class \(\H\), then \(\inf\nolimits_{h\in\H}V_\infty^{\ULA}(h)=0\) and these
bounds hold 
for the asymptotic variance itself.
\end{corollary}

\paragraph{Discussion}
It is well-known that if \(\hat f\) satisfies the so-called Poisson equation \(\PULA \hat f-\hat f=- f+\pi_{\gamma}^{\ULA}(f),\) then by taking \(g^\star=\hat f-\PULA\hat f\) as a control variate, we get \(\pi_{\gamma}^{\ULA}(f-g^\star)=\pi_{\gamma}^{\ULA}(f) \) and \(V_\infty^{\ULA}(f-g^\star)=0.\) The property \(h^\star=f-g^\star=\pi_{\gamma}^{\ULA}(f)\in \H\) can be achieved by taking, for example, \(\H\) to be a ball in a Sobolev space. Namely, let
$W^{s}_2 =  \bigl\{ h \in L^2(\lambda) : \,  {D}^{\alpha}h \in L^2(\lambda), \ \forall |\alpha| \leqslant s \bigr\}$
be the Sobolev space;
here $\lambda$ is the Lebesgue measure on \(\mathbb{R}^d\),
$\alpha = (\alpha_1,\ldots,\alpha_d)$ is a multi-index with $|\alpha|=\alpha_1+\ldots+\alpha_d$,
and $D^{\alpha}$ stands for the differential operator
${D}^{\alpha} = {\partial^{|\alpha|}}/{\partial x_1^{\alpha_1}\ldots\partial x_d^{\alpha_d}}$.
The weighted Sobolev space $W^{s}_2(\langle x \rangle^{\beta})$, $\beta \in \mathbb{R}$,
for a polynomial weighting function $\langle x \rangle^{\beta}= \bigl(1+\|x\|^2\bigr)^{\beta/2}$ is defined by
$W^{s}_2(\langle x \rangle^{\beta} ) = \bigl\{ h:\, h \cdot \langle x \rangle^{\beta} \in W^{s}_2 \bigr\}$.
Let $\H$  be a norm-bounded subset of $W^{s}_2(\langle x \rangle^{\beta})$ with {$\beta \in \rset$} and $s - d/2 > 0$.
Suppose also that $\|\langle x \rangle^{\alpha - \beta}\|_{\ltwo(\pi_{\gamma}^{\ULA})} < \infty$
for some $\alpha>0$.
Then  \(|\H_\eps|\lesssim \exp(\eps^{-d/s}),\) provided that \(\alpha>s-d/2\), see \cite[Corollary~4]{nickl2007bracketing}. Note that  \(h^\star=\pi_{\gamma}^{\ULA}(f)\in W^{s}_2(\langle x \rangle^{\beta}) \) for any \(s>0\) and any \(\beta<-1\) so that we can take \(\H\) as a norm-bounded subset of \( W^{s}_2(\langle x \rangle^{\beta})\) for arbitrary large \(s>0\).
Since \(\pi_{\gamma}^{\ULA}\) and all its derivatives have exponentially decaying tails (see  \cite{menozzi2010some}), $\|\langle x \rangle^{\alpha - \beta}\|_{\ltwo(\pi_{\gamma}^{\ULA})} < \infty$ for any \(\alpha>0\) and one can achieve that  \(|\H_\eps|\lesssim\exp(\eps^{-\delta})\) for arbitrary small \(\delta>0\) and at the same time  \(h^\star\in W^{s}_2(\langle x \rangle^{\beta})\). Practically one can use Stein control variates of the form \eqref{eq:stein} with infinitely smooth and compactly supported functions \(\phi\). This will guarantee that \(f-g_\phi\in  W^{s}_2(\langle x \rangle^{\beta})\) for some \(s>0\), provided that \(U\) is smooth enough and \(f\in W^{s}_2(\langle x \rangle^{\beta})\).
\subsection{Extension to the Stochastic Gradient Langevin Dynamics\label{sec:esvm-stochgrad}}
In this section, we shall consider the situations where the target $\pi$ is given by the posterior distribution in the Bayesian inference problem, that is, $\pi(\theta) \propto \exp{(-U(\theta))},$ where $U(\theta) = U_0(\theta) + \sum\nolimits_{i=1}^{\SSize}U_i(\theta)$ with $\SSize$ being a number of observations. Computing $\nabla U(\theta)$ requires a computational budget that scales linearly with $\SSize$. Hence it is often impossible to apply procedures based on discretisation of Langevin Dinamics 
directly. One possible solution advocated by \cite{welling:the:2011} is to replace $\nabla U(\theta)$ by an unbiased estimate. This gives rise to the SGLD algorithm, where the parameters are updated according to
\begin{equation}
\label{eq:SGLD}
\begin{split}
 \theta_{k+1} &=\theta_k - \gamma G(\theta_k,S_{k+1}) + \sqrt{2\gamma}\,\xi_{k+1}, \\
 G(\theta, S) &= \nabla{U}_0 (\theta) + \SSize\batchsize^{-1} \sum\nolimits_{i \in S} \nabla U_i(\theta),
 \end{split}
\end{equation}
where each $S_{k+1}$ is a random batch taking values in $\Sset_\batchsize$ (here $\Sset_{\batchsize}$ is the set of all subsets $S$ of $\{1, \ldots, \SSize\}$ with  $|S| = \batchsize$) which is sampled from a uniform distribution over $\Sset_\batchsize$ independently of $\mcf_k$ (here $(\mcf_k)_{k \geq 0}$ is the filtration generated by $\{(\theta_{\ell},S_{\ell})\}_{\ell \geq 0}$). Note that $\E[G(\theta_k,S_{k+1})|\mcf_k] = \nabla U(\theta_k)$ and therefore $G(\theta_k,S_{k+1})$ is an unbiased estimate of $\nabla U(\theta_k)$. The available variance reduction techniques for SGLD usually replace the stochastic gradient in \eqref{eq:SGLD} with more sophisticated estimates which preserve unbiasedness but have lower variance.
\par
The simplest variance reduction technique is the fixed-point method (SGLD-FP) proposed in \cite{baker2019control}. This method is applicable when the posterior distribution is strongly log-concave. We set $\hat{\theta}\in \Theta$ to be a fixed value of the parameter, typically chosen to be close to the mode of posterior distribution. We estimate the gradient $\nabla U(\theta)$ by
\begin{equation}
\label{eq:gradientupdate-FP}
G_{\FP}(\theta,S) = \nabla U_0(\theta)  + \SSize \batchsize^{-1} \sum\nolimits_{i \in S} \bigl(\nabla U_i(\theta) - \nabla U_i(\hat{\theta}) \bigr) + \sum\nolimits_{i=1}^{\SSize} \nabla U_i(\hat{\theta}).
\end{equation}
The SGLD-FP algorithm is obtained by plugging this approximation  into  \eqref{eq:SGLD}.
\par
More sophisticated variance reduction methods typically use reference values $(g_k^i)_{i=1}^{\SSize}$ of the gradient  $(\nabla U_i)_{i=1}^{\SSize}$   from previous iterates (and not only the last iterate); as a result, constructed sequence $(\theta_k)_{k=0}^\infty$ is often not Markovian. One particular example is SAGA-LD method, adapted from \cite{leroux:schmidt:bach:2012,defazio2014saga}. If $i \in S_{k}$, the reference value is updated, that is, $g_{k+1}^i = \nabla U_i(\theta_k)$. Otherwise, the reference value is simply propagated, that is, $g_{k+1}^i= g_k^i$. One then considers the following gradient estimator
\begin{equation}
\label{eq:gradientupdate}
G_{\SAGA}^k(\theta,S)  = \nabla U_0(\theta)  + \SSize M^{-1} \sum\nolimits_{i \in S} \bigl( \nabla U_i(\theta) - g_k^i \bigr) + g_k \eqsp, \quad g_k = \sum\nolimits_{i=1}^{\SSize} g_k^i \,.
\end{equation}
The recursion is initialized  with \(g^i_0=\nabla U_i(\theta_0),\) \(i \in \{1,\ldots, \SSize\},\) and $g_0= \sum_{i=1}^{\SSize} g_0^i.$ Finally, the gradient is computed according to \eqref{eq:gradientupdate} and plugged into \eqref{eq:SGLD}.
\par
For theoretical analysis of SGLD and SGLD-FP algorithms we need the following assumptions on $U$. Without loss of generality, we consider only SGLD; the same reasoning applies to SGLD-FP. 
\begin{description}
\vspace{0.5em}
\item[(SGLD)\label{SGLD}]
The function $U(\theta) = U_0(\theta) + \sum\nolimits_{i=1}^{\SSize}U_i(\theta)$
satisfies the following conditions.
\begin{enumerate}[1)]
\item Lipschitz gradient: for any $i \in \{0, \ldots, \SSize\}$, $U_i$ is continuously differentiable on \(\rset^d\) with $\tlipU$-Lipschitz gradient;
\item Convexity: for any $i \in \{0, \ldots, \SSize\}$, $U_i$ is convex;
\item\label{item:strong-convex} Strong convexity: there exists a constant $\mU > 0$, such that for any $\theta, \theta^{\prime} \in \rset^d$ it holds that
$U(\theta') \geq U(\theta) + \pscal{\nabla U(\theta)}{\theta'-\theta} + (\mU/2)\|\theta'-\theta\|^2$.
\end{enumerate}
\vspace{0.5em}
\end{description}

Note that using Stein control variates with SGLD-based sampling procedure \eqref{eq:SGLD} eliminates benefits of using $G(\theta, S)$ instead of exact gradient $\nabla U(\theta)$. Following \cite{friel:mira:oates:2015}, we replace $\nabla U$ by its stochastic counterpart. More precisely, for $k$-th iteration of SGLD algorithm, we consider the control variates of the form
\begin{equation}
\label{eq:stein-stochastic}
g_\phi(\theta, S) = - \pscal{\phi(\theta)}{G(\theta,S)} + \divergence\bigl(\phi(\theta)\bigr).
\end{equation}
The control variate \(g_{\phi}\) depends now on the pair $(\theta, S)$. Let \(\H = \{f(\theta) - g_{\phi}(x): \phi \in \Phi\}\), where $x = (\theta, S) \in \Xset = \Theta \times \Sset_\batchsize$.   Consider another sequence $\bigl(\tilde{S}_k\bigr)_{k=0}^\infty$ of independent batches uniformly distributed over $\Sset_\batchsize$ such that for any $k$, $\tilde S_k$ is independent of $\mathcal F_k$. Denote by $P_{\mathsf{SGLD}}$ the transition kernel of SGLD and let $\Upsilon_\batchsize$ be a uniform distribution over $\Sset_\batchsize$. Set $\Xkernel : = P_{\mathsf{SGLD}} \otimes \Upsilon_\batchsize$ and  $X_k = (\theta_k, \tilde S_{k})$.
\begin{proposition}
\label{lem:SGLDW2contraction}
Assume \ref{SGLD}. Then for any step size $\gamma \in \bigl(0, \tlipU^{-1}(\SSize+1)^{-1}\bigr)$, $\Xkernel$ satisfies \ref{WE}-2 with $\wascoef_2 = \sqrt{1 - \gamma \mU}$ and  $\metric(x,x') = \| \vartheta- \vartheta'\| + \indiacc{S \neq S'}$ for any $x = (\vartheta,S)$ and $x' = (\vartheta',S')$. Moreover, $\Xkernel$ has a unique invariant measure $\overline \pi = \pi_{\gamma}^{(\mathsf{SGLD})} \otimes \Upsilon_\batchsize$.
\end{proposition}
\begin{proof}
See \Cref{sec:SGLD_contraction}.
\end{proof}
\par
Similarly to Langevin Dynamics, we define
\begin{equation*}
\label{eq:Vinf_SGLD}
    V_\infty^{\SGLD}(h) := \sum\nolimits_{\ell \in \zset} \E_{\overline{\pi}}\Bigl[
    \bigl(h(X_0) - \overline{\pi}(f)\bigr)
    \bigl(h(X_{|\ell|}) - \overline{\pi}(f)\bigr)
    \Bigr].
\end{equation*}
\begin{theorem}
\label{thm:sgld}
Let $\H \subseteq \Lip_{b,\metric}(\Lnorm, \B)$ and assume that
\ref{SGLD} holds.
Fix any $\gamma \in \bigl(0, \tlipU^{-1}(\SSize+1)^{-1}\bigr)$ and set $b_n= 2 \lceil\log(n)/\log(1/\wascoef_1) \rceil$ with $\wascoef_1 = \sqrt{1 - \gamma\mU}$. Then, for any $\eps > 0$ and $\delta\in(0,1)$, with probability at least $1-\delta$,
\begin{multline*}
    V_\infty^{\SGLD}(\hh[n,\eps]) - \inf\nolimits_{h\in\H}V_\infty^{\SGLD}(h) \\
    \lesssim
    \C_4 \, \eps\log(n) + 
    \C_5\sqrt{\frac{\log^{5}(n)}{n}}\biggl(\frac{|\H_{\eps}|}{\delta}\biggr)^{1/\log(n)}
    + \C_6 \, \frac{\log{n}}{n},
\end{multline*}
where
\begin{align*}
\C_4 = \frac{\sqrt{\CM} + \sqrt{\CH}}{\mU \gamma}, \
\C_5 = \frac{B^2 R_1(L,\xi)}{(\mU\gamma)^2} + \frac{B^2 R_2(L,\xi)}{(\mU \gamma)^{4 + 2/\log{n}}} + \frac{\sqrt{\CM \CH}}{\mU \gamma},\ 
\C_6 = \frac{\CH(\mU \gamma)+ \varsigma}{(\mU \gamma)^2} 
\end{align*}
with 
$\CM$, $\varsigma$ from \Cref{prop:W1_concentration}, $\CH = \sup\nolimits_{h \in \H} \PVar[\pi_{\gamma}^{\SGLD}](h)$, and  constants $R_1(L,\xi)$, $R_2(L,\xi)$ which can be tracked from \eqref{eq:rosenthal_constants}.
\end{theorem}
\begin{proof} By \cref{lem:SGLDW2contraction}, \ref{WE}-2 holds with $\wascoef_2 = \sqrt{1-\gamma \mU}$, and, by Lyapunov inequality, \ref{WE}-1 also holds with $\wascoef_1=\wascoef_2$.
Hence, the second part of \cref{prop:W1_concentration} can be applied with $p = \log{n}$. The remaining part follows from \cref{th:main} with computation of the inverse function in the right-hand side of \eqref{eq:rosenthal}.
\end{proof}
\begin{corollary}\label{cor:sgld}
Under the assumptions of \Cref{thm:sgld}, if class \(\H\) is parametric, that is, \(|\H_\eps|\leq C_{\rho}\eps^{-\rho}\) for all \(\eps\in(0,1)\) and some constants \(C_\rho,\rho>0.\) Then it holds with probability at least \(1-1/n\),
\[
 V_\infty^{(\mathsf{SGLD})}(\hh[n,\eps]) - \inf\nolimits_{h\in\H}V_\infty^{(\mathsf{SGLD})}(h)\lesssim n^{-1/2} \log^{5/2}(n) ,
\]
where \(\lesssim\) stands for inequality up to a constant depending on \(\rho\) and other constants from \Cref{thm:sgld}.
Moreover, if additionally \(\overline{\pi}(f)\in \H \), then \(\inf\nolimits_{h\in\H}V_\infty^{\SGLD}(h)=0\) and these
bounds hold 
for the asymptotic variance itself.
\end{corollary}
\begin{remark}
If the class \(\H\) is constructed using  Stein control variates, we can ensure the inclusion $\H \subseteq \Lip_{b,\metric}(\Lnorm, \B)$ by taking smooth and compactly supported functions \(\phi\).
This in turn can be achieved by multiplying a given smooth function \(\phi\) with a mollifier function, that is, an infinitely smooth compactly supported function.  
\end{remark}

\section{Experiments}\label{sec:numerics}
In this section, we numerically compare the following two methods to choose control variates: Empirical Variance Minimisation (EVM) method, where a control variate is determined by minimizing the marginal variance, see \eqref{eq:evm},
and Empirical Spectral Variance Minimisation (ESVM) method,
where a control variate is determined by minimizing the spectral variance, see \eqref{eq:sv}. Implementation is available at \url{https://github.com/svsamsonov/vr_sg_mcmc}.



\subsection{Toy example\label{pp:toy_example}}
We first consider a multimodal distribution in $\rset^2$ from \cite{rezende2015NF}. Namely, let $\pi(x_1,x_2)= Z^{-1} \rme^{-U(x_1,x_2)}$, where $Z$
is the normalization constant and
\[
U(x_1,x_2) = 
\frac{(\|x\| - \mu)^2}{2M^2} 
- \log\Bigl(\rme^{-(x_1-\mu)^2/2\sigma^2} + \rme^{-(x_1+\mu)^2/2\sigma^2}\Bigr).
\]
We choose $M=1$ and $\mu = \sigma = 3$; the respective density profile is presented in \Cref{fig:te}. Our aim is to estimate $\pi(f)$ with $f(x_1,x_2) = x_1+x_2$ {using ULA}.
The parametric class $g_{\varphi}$ in \eqref{eq:stein} is generated by $\varphi(x) = \sum\nolimits_{k=1}^{p}\beta_{k}\psi_{k}(x)$, where $\psi_k = e^{-\|x-\mu_k\|^2/2\sigma_{\psi}^2}$ with all $\mu_k$ regularly spaced in $[-3,3] \times [-3,3]$ and $\sigma_{\psi} = 2$. Boxplots displaing variation of $100$ estimates for EVM and ESVM are presented in the same \Cref{fig:te}. 
Furthermore, we compute sample autocovariance functions for a trajectory with and without adding ESVM and EVM control variates. The results reflect a spectacular decrease in high-order autocovariance for ESVM, see \Cref{fig:te}. Note that EVM aims at minimizing only the lag-zero autocovariance, that is why the autocovariance function for ESVM-adjusted trajectory decreases much faster.

\begin{table}[!htp]\centering
\resizebox{0.75\textwidth}{!}{\centering
\begin{tabular}{@{}l c c c c @{}}
\toprule
Experiment& $n_{\text{burn}}$ &  $n_{\text{test}}$ & $\gamma$ & batch size\\ \toprule
Toy example, \Cref{pp:toy_example} & $10^3$ &$10^4$ & $0.1$ & - \\
Gaussian Mixture, \Cref{pp:gmm} & $10^4$ &$10^5$ & $0.01$ & 10 \\
\bottomrule
\end{tabular}}
\vspace{0.5em}
\caption{Experimental hyperparameters}\label{tab:Table_setup}
\vspace{-2.2em}
\end{table}
\begin{figure}[htbp]
\vspace{-1em}
\includegraphics[width=0.27\linewidth]{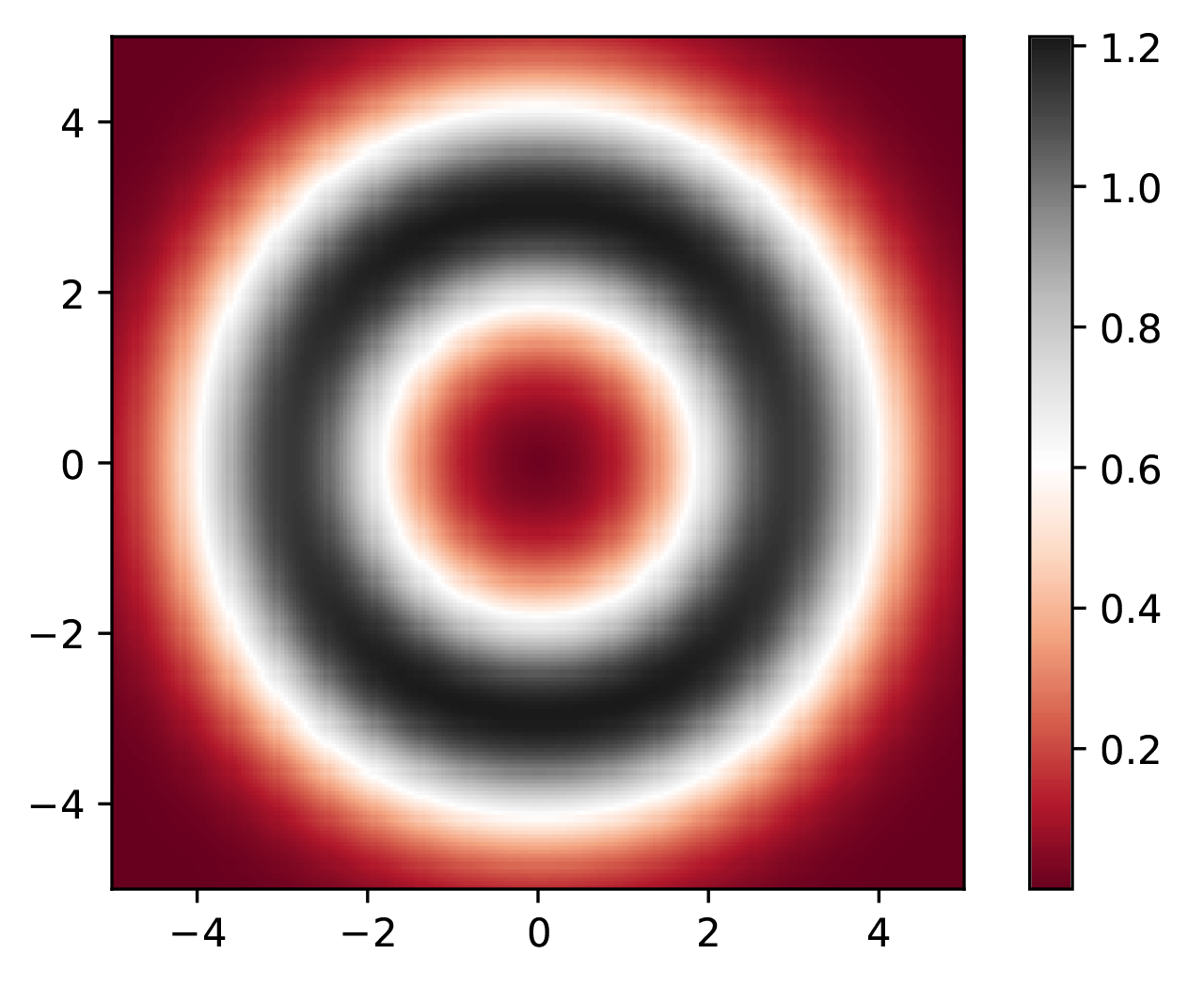}
\includegraphics[width=0.33\linewidth]{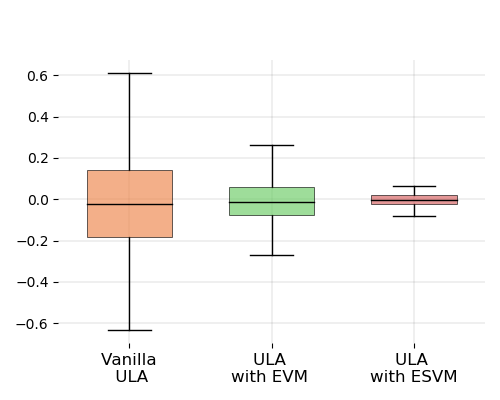}
\includegraphics[width=0.36\linewidth]{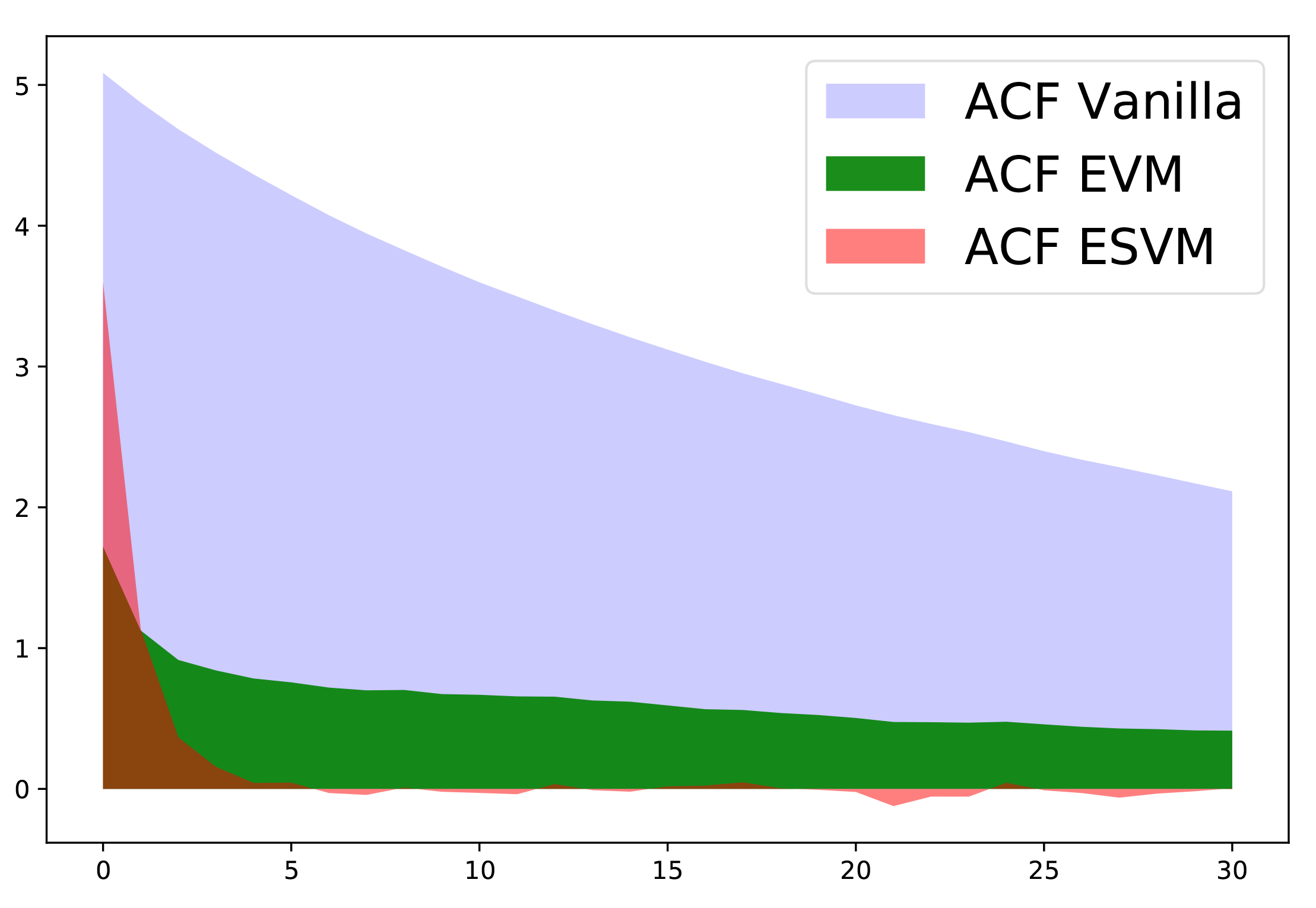}
\vspace{-0.5em}
\caption{\label{fig:te} Toy example from \Cref{pp:toy_example}. From left to right: (1)~density profile, (2)~boxplots displaing variation of $100$ estimates for vanilla ULA, ULA with EVM, and ULA with ESVM, (3)~sample autocovariance functions a trajectory with and without ESVM and EVM.}
\vspace{-0.5em}
\end{figure}


\subsection{Gaussian mixture model\label{pp:gmm}} 
We consider posterior mean estimation for unknown parameter $\mu$ in a Bayesian setup with normal prior $\mu \sim \mathcal{N}(0,\sigma_{\mu}^2)$, $\sigma^2_{\mu} = 100$,
and sample $(X_k)_{k=0}^{\SSize-1}$, $\SSize = 100$, drawn from the Gaussian mixture model
\[
0.5\,\mathcal{N}(-\mu,\sigma^2) + 0.5\,\mathcal{N}(\mu,\sigma^2)
\quad\text{with}\
\mu = 1, \ \sigma^2 = 1.
\]
The density of the posterior distribution over $\mu$ 
is given in \Cref{fig:gmm}. 
It has $2$ modes roughly corresponding to $\mu = 1$ and $\mu = -1$. To generate data from this posterior distribution and estimate posterior mean, we use SGLD. 
The parametric class $g_{\varphi}$ in \eqref{eq:stein-stochastic} is generated by $\varphi(x) = \beta_0x^2 + \beta_1x + \beta_2$. Boxplots displaing variation of $100$ estimates for EVM and ESVM and respective sample autocovariance functions 
are also presented in \Cref{fig:gmm}. 
 Note that the increase in lag-zero autocovariance for ESVM is explained by the additional randomness in \eqref{eq:stein-stochastic}. On contrary, EVM favors far too small coefficients to overcome this additional randomness, which leads to poor variance reduction.
\par

\begin{figure}[htbp]
\vspace{-1em}
\includegraphics[width=0.28\linewidth]{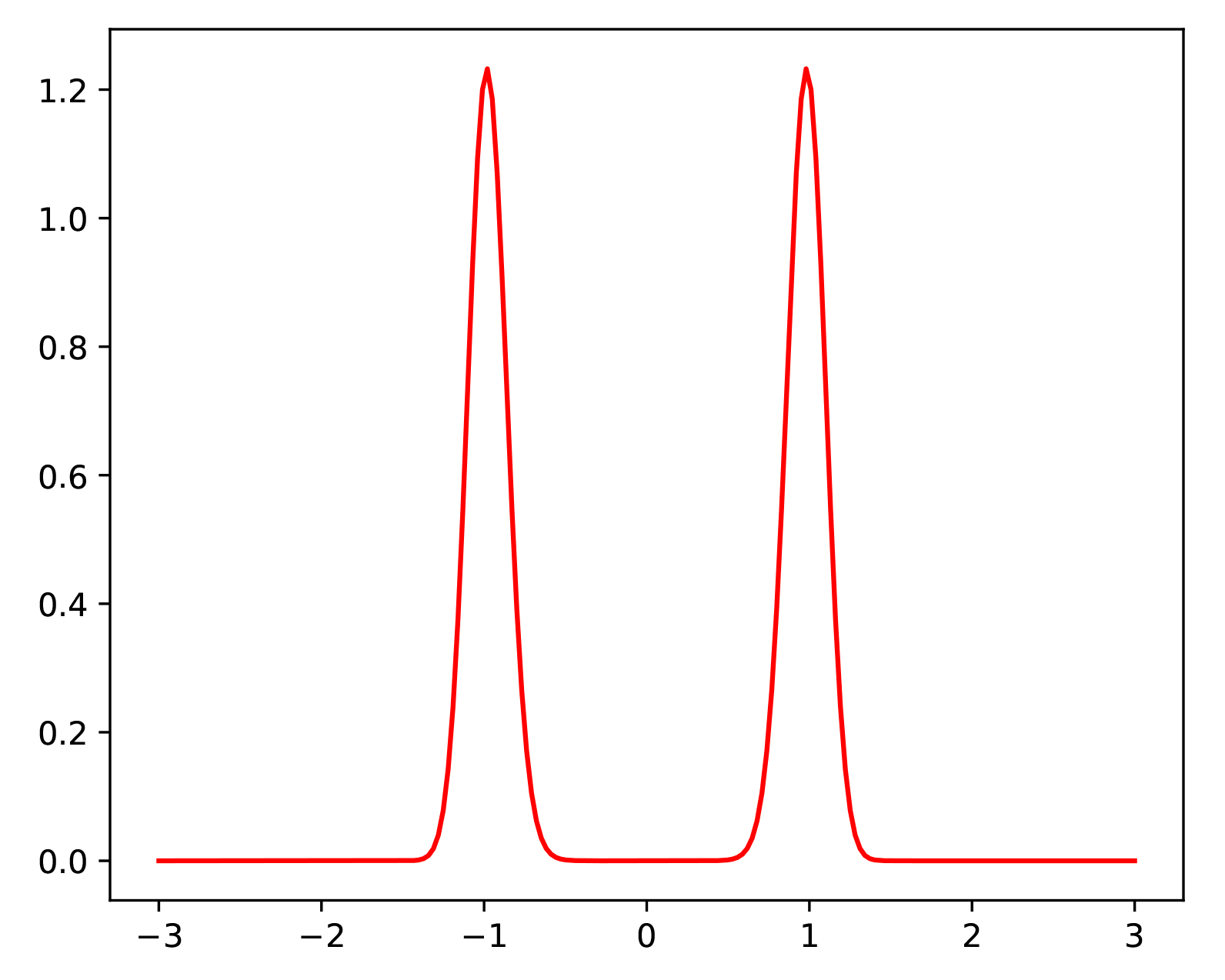}
\includegraphics[width=0.33\linewidth]{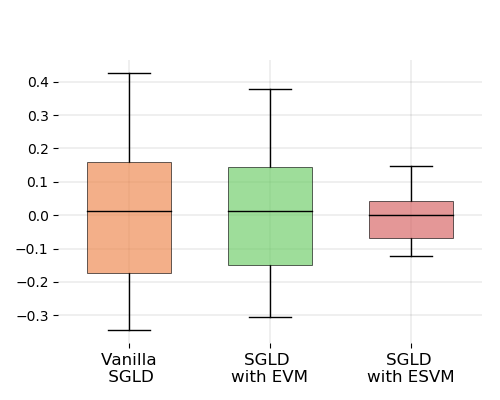}
\includegraphics[width=0.37\linewidth]{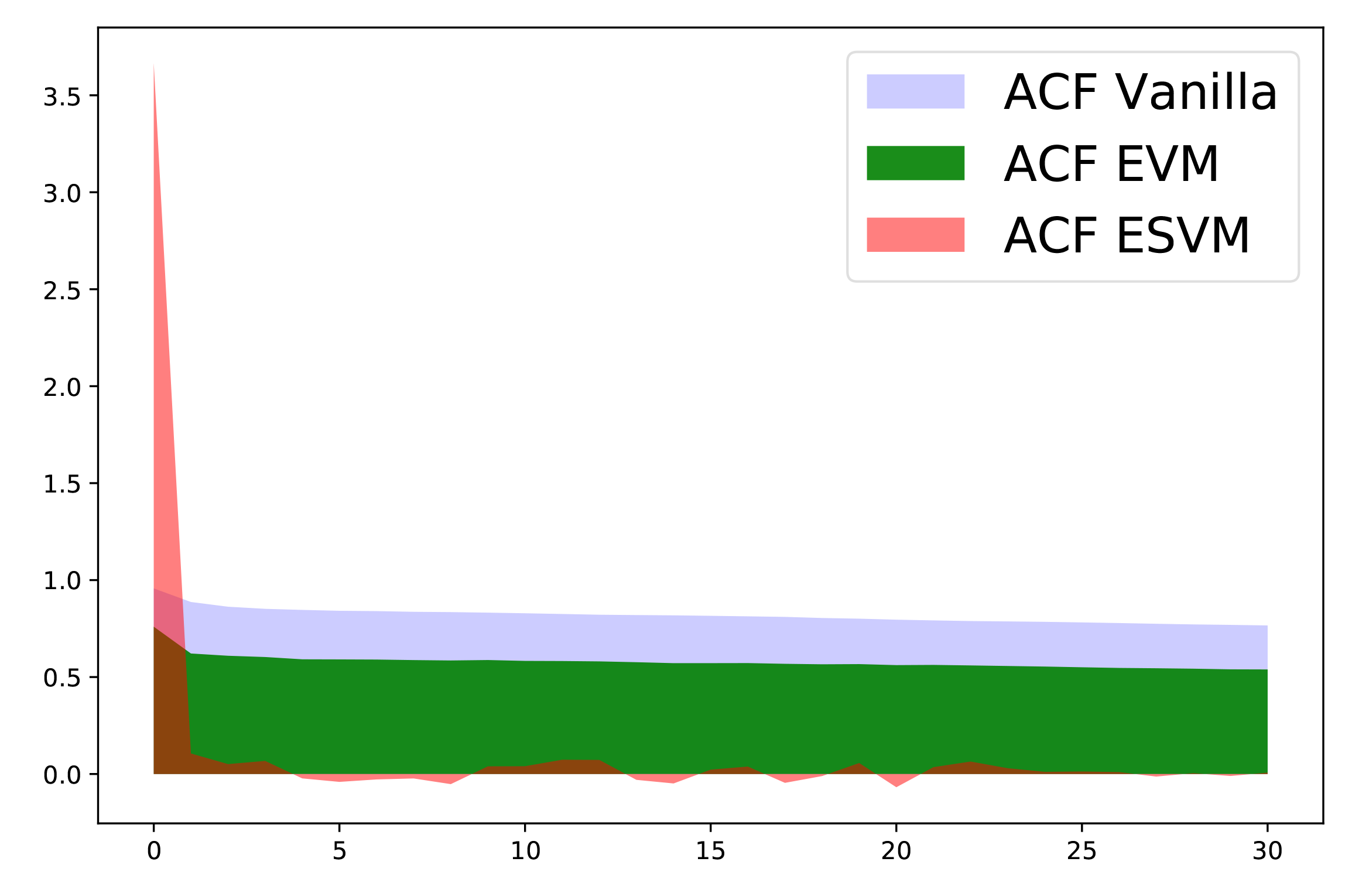}
\vspace{-1.5em}
\caption{\label{fig:gmm} 
Gaussian mixture model from \Cref{pp:gmm}. From left to right: (1)~density of the posterior distribution, (2)~boxplots displaing variation of $100$ estimates for vanilla SGLD, SGLD with EVM, and SGLD with ESVM, (3)~sample autocovariance functions for a trajectory with and without ESVM and EVM.}
\vspace{-0.5em}
\end{figure}

%

\subsection{Bayesian logistic regression}\label{pp:blr}
The probability of the $i$-th output $y_i\in \{-1,1\}$, $i=1,\ldots,\SSize$, is given by
$\dens[]{y_i\vert \bx_i, \theta} = (1+\rme^{-y_i \pscal{\theta}{\bx_i}})^{-1}$, where $\bx_i$ is a $d \times 1$ vector of predictors and $\theta$ is the vector of unknown
regression coefficients. We complete the Bayesian model by considering the  Zellner $g$-prior \(\mathcal{N}_{d}(0,g(\bX^{\T}\bX)^{-1})\) for $\theta$ where $\bX= [\bx_1,\dots,\bx_N]$ is an $\SSize \times d$ design matrix, see \cite[Section~2]{hanson2014informative}. Normalizing the covariates, for $\tilde{\bx}_i= (\bX^{\T} \bX)^{-1/2} \bx_i$ and $\tilde{\theta}= (\bX^{\T} \bX)^{1/2} \theta$, we get $\pscal{\theta}{\bx_i} = \pscal{\tilde{\theta}}{\tilde{\bx}_i}$, under the Zellner $g$-prior, $\tilde{\theta} \sim \mathcal{N}_{d}(0,g\bI_d)$.
\par
We analyse the performance of EVM and ESVM methods on  two datasets from the UCI repository. The first
dataset, EEG, contains $\SSize=14\,980$ observations in dimension $d=15$, the second dataset, SUSY, has $\SSize=500\,000$ observations in dimension $d=19$.
The data is first split into a training set $\mathcal{T}_N^{\operatorname{train}}= \{(y_i,\bx_i)\}_{i=1}^\SSize$ and
a test set $\mathcal{T}^{\operatorname{test}}_K= \{(y'_i,\bx'_i)\}_{i=1}^{K}$ by randomly picking $K=100$ test points from the data.
We use the SGLD-FP and SAGA-LD algorithms to approximately sample from the posterior distribution
\(\dens[]{\tilde{\theta}| \mathcal{T}^{\operatorname{train}}_N} \). Given a sample $(\tilde{\theta}_k)_{k=0}^{n-1}$, we can estimate
the predictive distribution for a fixed test point $(y',\bx')$, that is,
\(\dens[]{y' | \bx'} = \int\nolimits_{\mathbb{R}^d}\dens[]{y' |\bx',\tilde{\theta}}\dens[]{\tilde{\theta} | \mathcal{T}^{\operatorname{train}}_N} \, \rmd \tilde{\theta}\),
by computing the ergodic mean $n^{-1} \sum_{k=0}^{n-1} f(\tilde{\theta}_k)$ for
$f(\tilde{\theta})=\dens[]{y' | \bx',\tilde{\theta}}$.
To get rid of randomness caused by the random choice of a test point, we estimate the
average predictive distribution for the whole test set $\mathcal{T}^{\operatorname{test}}_K$
by computing the ergodic mean  for the function $f(\tilde{\theta})= K^{-1} \sum_{i=1}^{K} \dens[]{y'_i | \bx'_i,\tilde{\theta}}$.
Boxplots for the estimation of average predictive distribution
are shown in \Cref{fig:blr1}. 
Note that 
ESVM leads to
a significant variance reduction for both SGLD-FP and SAGA-LD.
\begin{figure}[htbp]
\minipage{0.5\textwidth}
\begin{rline}
\includegraphics[width=0.49\linewidth]{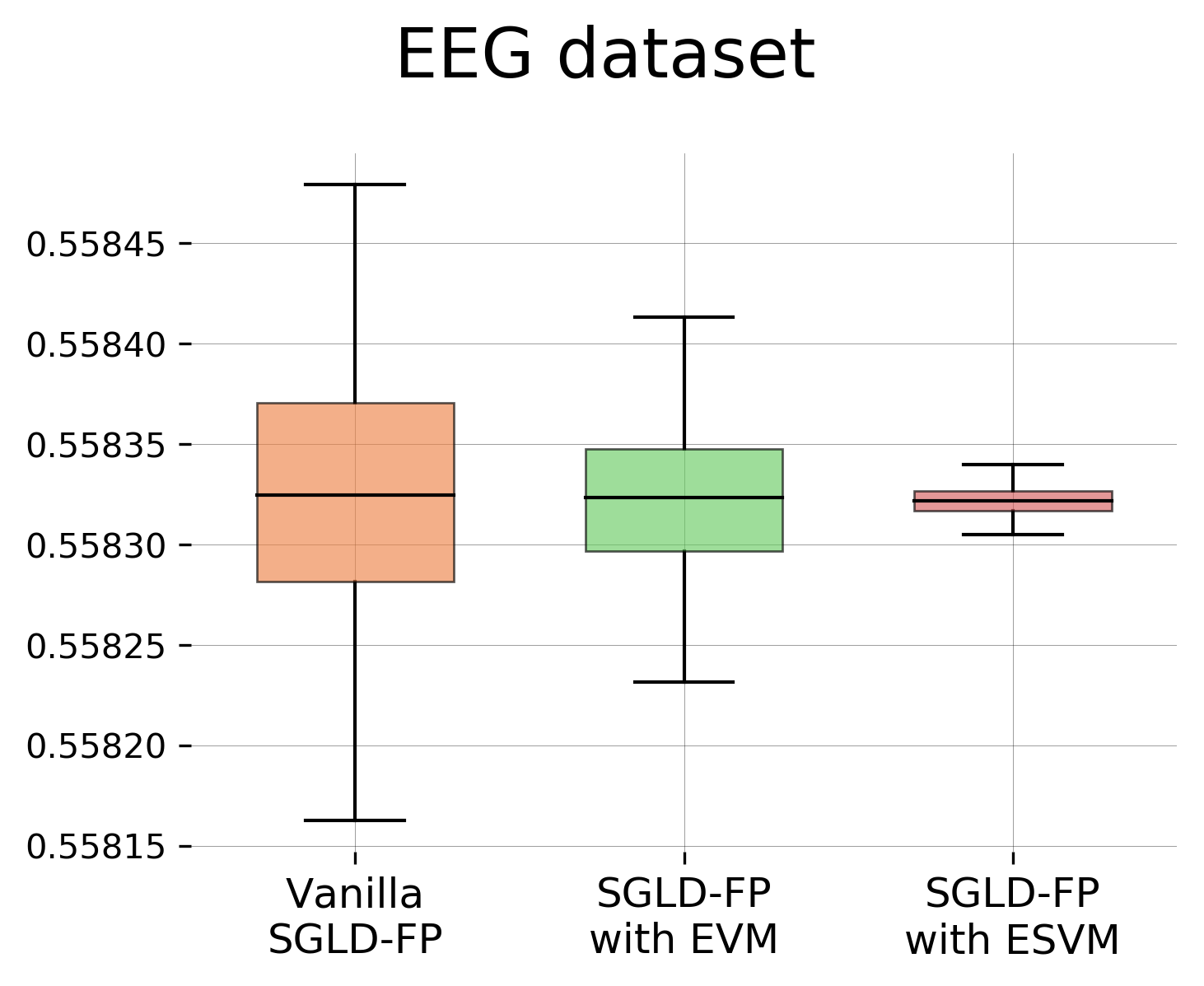}
\includegraphics[width=0.49\linewidth]{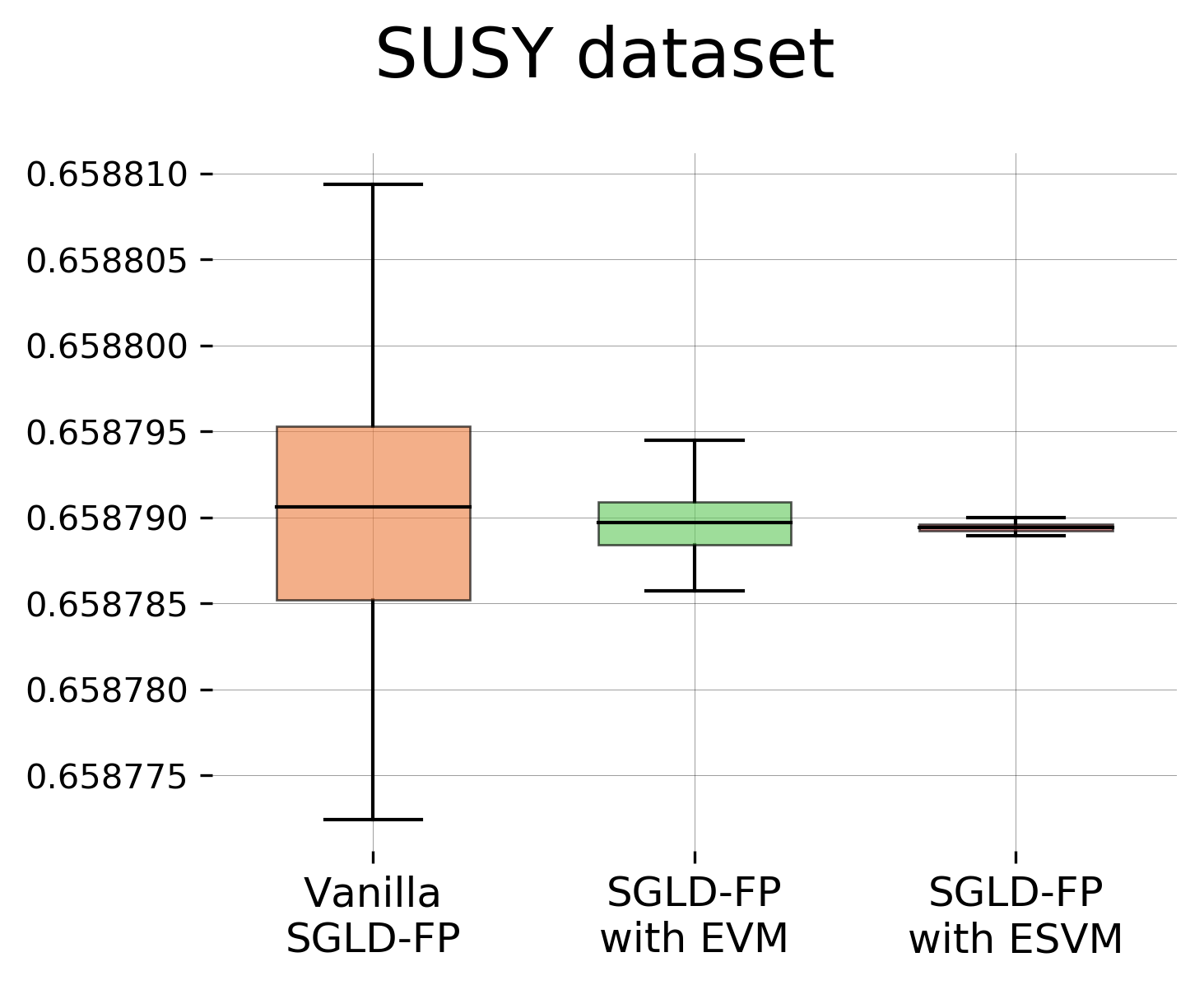}
\end{rline}
\endminipage\hfill
\minipage{0.49\textwidth}
\includegraphics[width=0.49\linewidth]{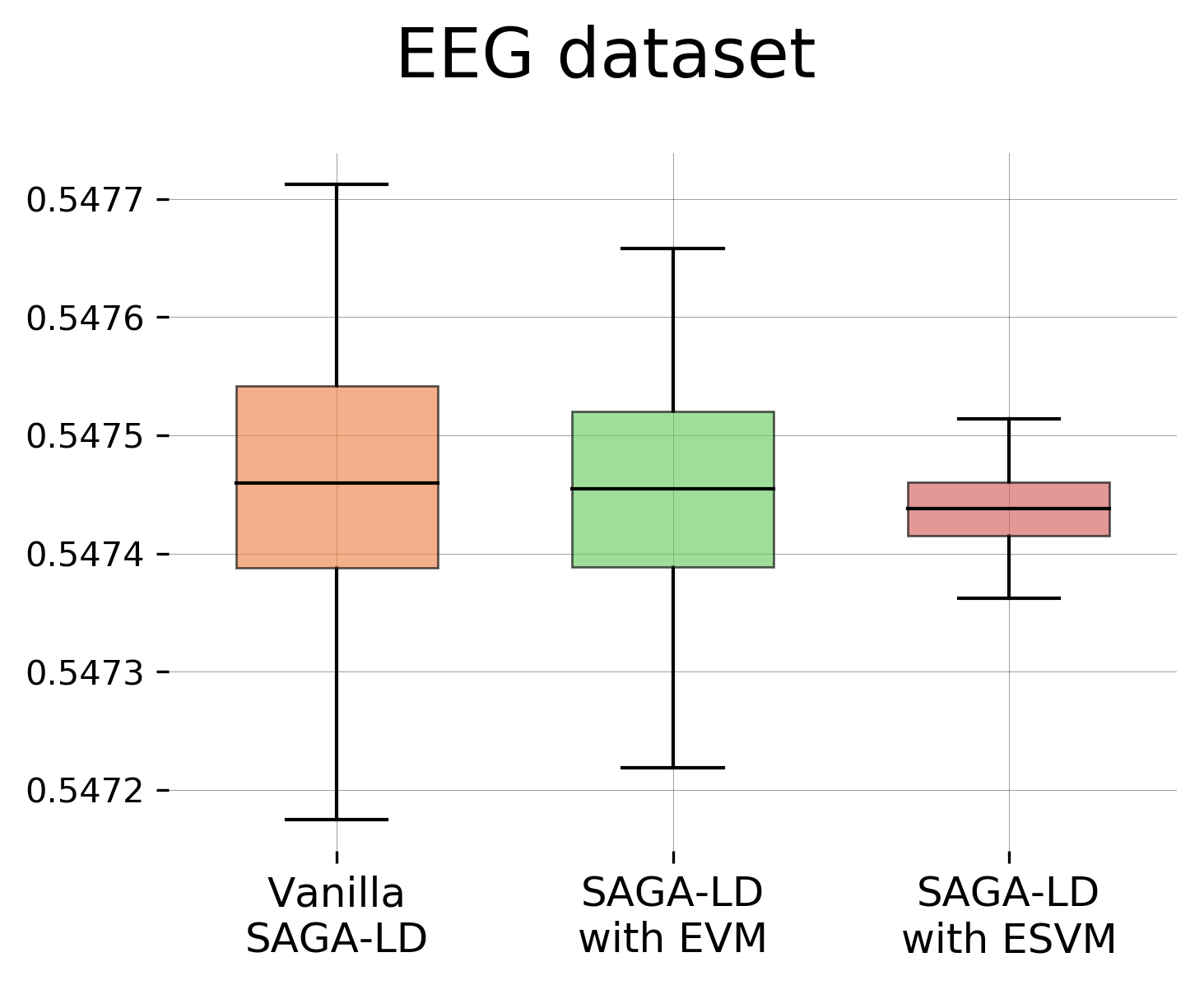}
\includegraphics[width=0.49\linewidth]{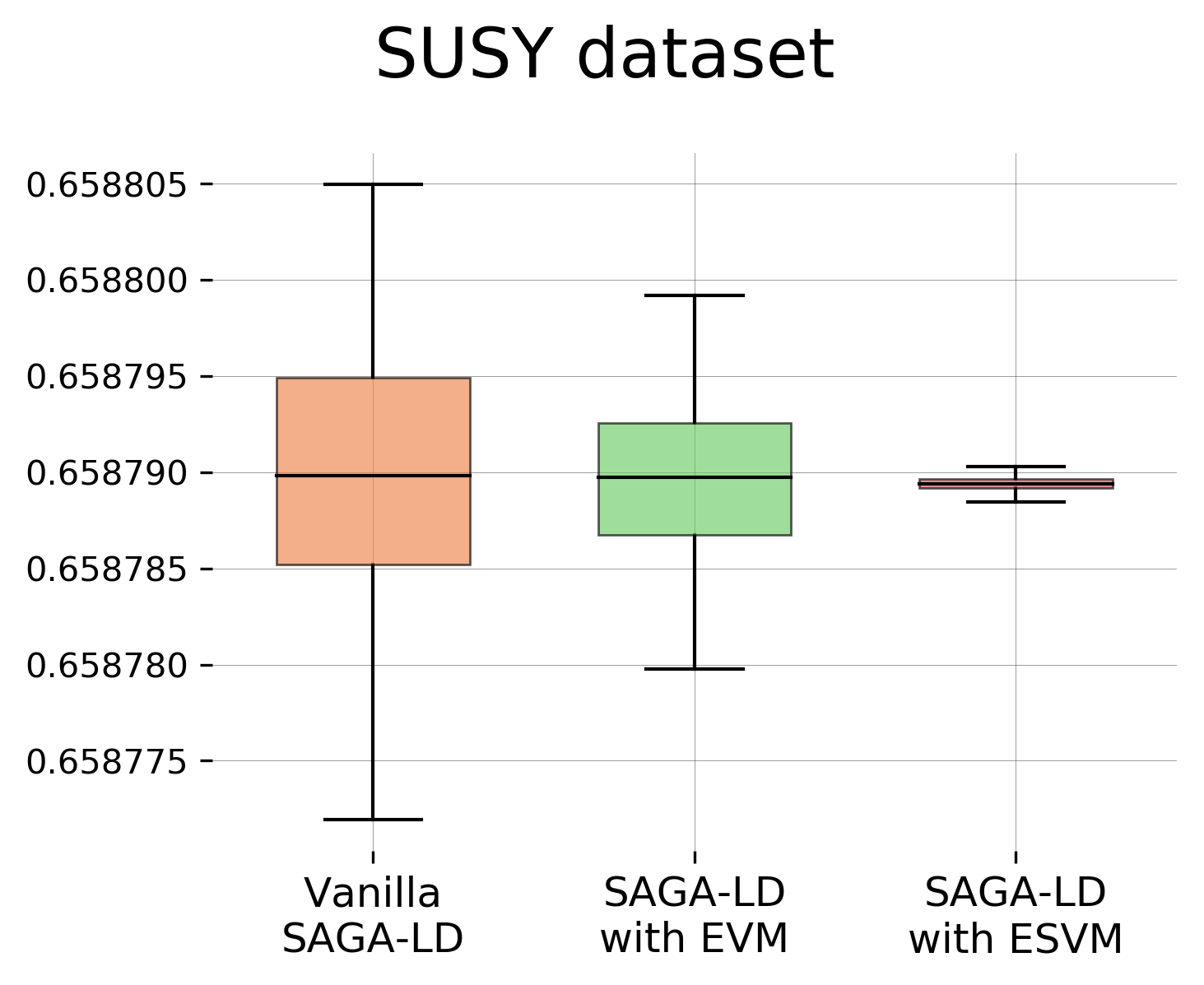}
\endminipage
\caption{\label{fig:blr1}
Bayesian logistic regression for EEG and SUSY datasets from \Cref{pp:blr}.
Boxplots displaing variation of $100$ estimates of average predictive distribution for 
(1) left panel: vanilla  SGLD-FP, SGLD-FP with EVM, and SGLD-FP with ESVM, (2) right panel: vanilla  SAGA-LD, SAGA-LD with EVM, and SAGA-LD with ESVM.}
\vspace{-0.5em}
\end{figure}
\par
Further, for the EEG dataset we plot in \Cref{fig:acf1} a part of the trajectory $f(\tilde{\theta}_m)= K^{-1} \sum_{i=1}^{K} \dens[]{y'_i | \bx'_i,\tilde{\theta}_m}$ for \(500\) consecutive sample values \(\tilde{\theta}_m\) with and without adding the ESVM control variate.
These trajectories are accompanied by the sample autocovariance functions for vanilla and variance-reduced samples for both EVM and ESVM. 
Again, since EVM aims at minimizing only lag-zero autocovariance, the decrease in autocovariance function for this method is smaller than for ESVM. We also report in \Cref{fig:acf_decrease_EEG} how autocovariance functions change with batch sizes. 
Note that for small batch sizes ESVM still manages to remove correlations, while EVM almost fails. At the same time, increasing the batch size leads to similar results for EVM and ESVM.

\begin{table}[!htp]\centering
\resizebox{0.8\textwidth}{!}{\centering
\begin{tabular}{@{}l c c c c c @{}}
\toprule
Experiment & $n_{\text{burn}}$ & $n_{\text{train}}$ & $n_{\text{test}}$ & $\gamma$ & batch size\\ \toprule
Logistic regression, EEG dataset & $10^4$ & $10^4$ & $10^5$ & $0.1$ & 15 \\
Logistic regression, SUSY dataset& $10^5$ & $10^5$ & $10^6$ & $0.1$ & 50 \\
\bottomrule
\end{tabular}}
\vspace{0.5em}
\caption{Experimental hyperparameters}\label{tab:Table_setup_lr}
\vspace{-1.5em}
\end{table}

\begin{figure}[htbp]
\includegraphics[width=\linewidth]{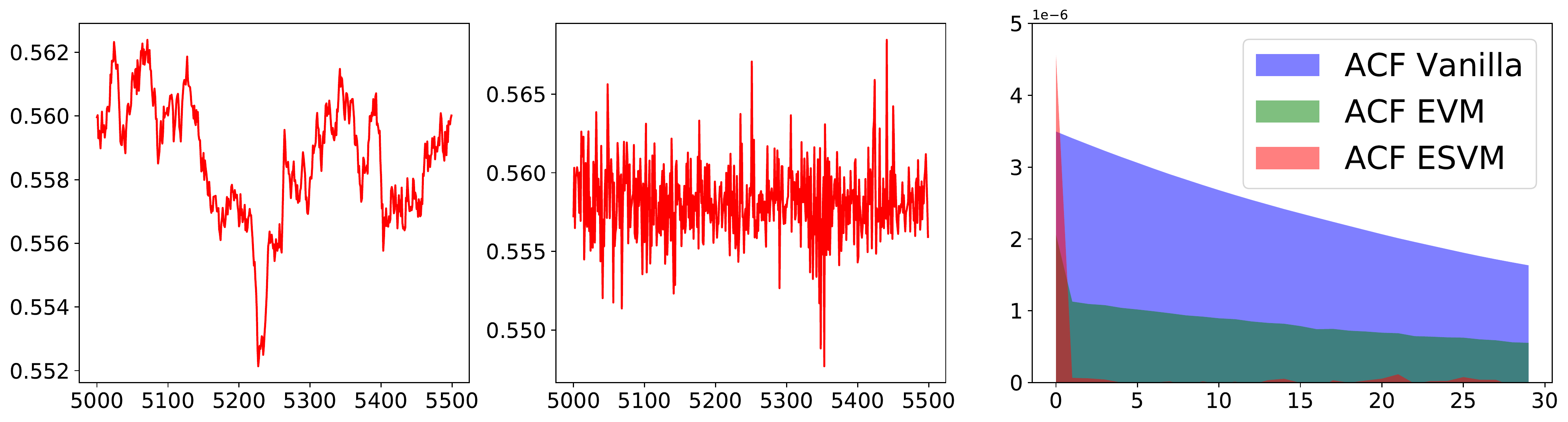}
\vspace{-1.5em}
\caption{\label{fig:acf1}
Bayesian logistic regression for the EEG dataset from \Cref{pp:blr}.
From left to right: 
(1)~part of a trajectory without ESVM,
(2)~part of a trajectory with ESVM,
(3)~sample  autocovariance  functions  for  a  trajectory  with  andwithout ESVM and EVM.}
\end{figure}
\begin{figure}[htbp]
\vspace{-0.5em}
\includegraphics[width=\linewidth]{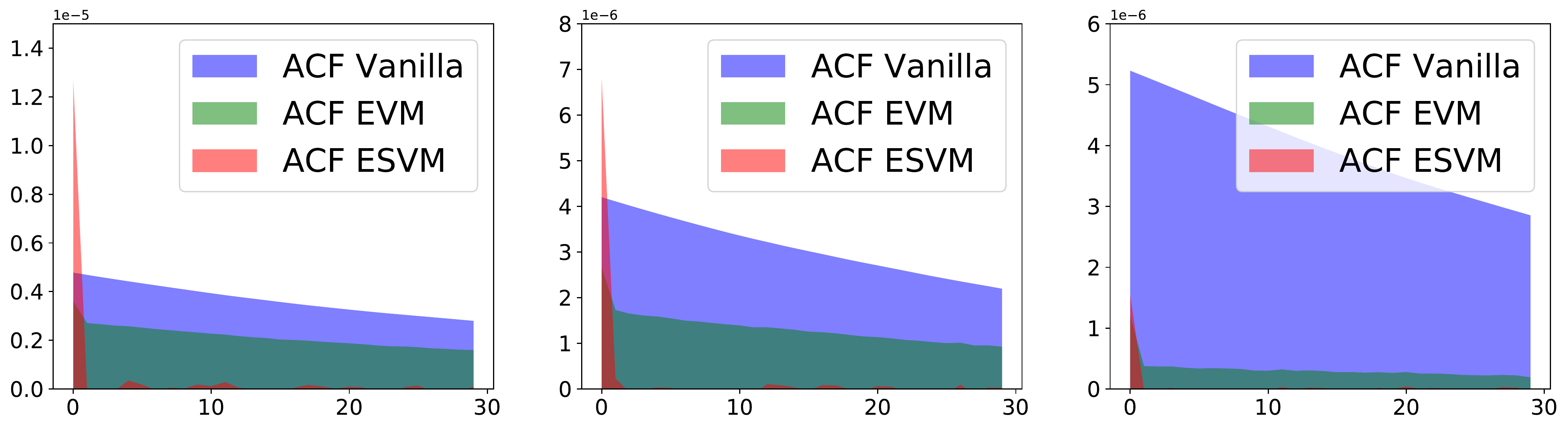}
\vspace{-1.5em}
\caption{\label{fig:acf_decrease_EEG}
Bayesian logistic regression for the EEG dataset from \Cref{pp:blr}. Comparison of sample autocovariance for different batch sizes.
From left to right: batch size 5, 15, 150 respectively.}
\vspace{-1em}
\end{figure}

\subsection{Bayesian Probabilistic Matrix Factorization}\label{pp:bpmf}
A typical problem in Recommendation Systems is to predict user's rating for a particular item given other user's ratings of this item and how a given user evaluated other items. A common approach to this problem is Probabilistic Matrix Factorization via Bayesian inference, see \cite{salakhutdinov:2008}. Namely, we are interested in approximating matrix $R \in \mathbb{R}^{M \times N}$, where $M$ is a number of users, $N$ is a number of rated items, and $R_{i,j}$ stands for rating assigned by $i$-th user to $j$-th item. Due to natural limitations (user is unlikely to rate all possible items), we observe only a some small subset of elements of $R$ and want to predict ratings of the hidden part. In Probabilistic Matrix Factorization, we aim at representing  $R$ as a product $R = U^{\T} V + C$, where $U \in \mathbb{R}^{D \times M}$, $V \in \mathbb{R}^{D \times N}$, and $C \in \mathbb{R}^{M \times N}$ being a matrix of biases with elements $C_{i,j} = a_i + b_j$, $a \in \mathbb{R}^{M}$, $b \in \mathbb{R}^{N}$. In the subsequent experiments we assume that rank parameter $D=10$ is fixed. The naive solution would be to find
\begin{align*}
    U,V,a,b = \argmin\nolimits_{\,U,V,a,b}\sum\nolimits_{(i,j) \in I_{\text{train}}}
\bigl(R_{i,j}- \pscal{U_i}{V_j} - a_i - b_j\bigr)^2,
\end{align*}
where $I_{\text{train}}$ is a train subset of ratings. Unfortunately, optimizing this criteria leads to significantly overfitted model. One possible approach to overcome overfitting is to consider penalised model
\begin{multline*}
    U,V,a,b = \argmin\nolimits_{\,U,V,a,b}\sum\nolimits_{(i,j) \in I_{\text{train}}}
    \bigl(R_{i,j}- \pscal{U_i}{V_j} - a_i - b_j\bigr)^2 \\+ \lambda_U \|U\|^2 + \lambda_V \|V\|^2 + \lambda_a \|a\|^2 + \lambda_b \|b\|^2,
\end{multline*}
but it requires careful tuning of penalisation coefficients $\lambda_U,\lambda_V,\lambda_a,\lambda_b$. We thus would benefit a lot from Bayesian approach for tuning weights; this was pointed out in \cite{salakhutdinov:2008}. We follow a slightly simplified formulation proposed by \cite{chen2014}, that is, we consider
\begin{align*}
&\lambda_U, \lambda_V, \lambda_a, \lambda_b \sim \Gamma(1,1),\quad
U_{k,i} \sim \mathcal{N}\bigl(0,\lambda_U^{-1}\bigr),\quad
V_{k,j} \sim \mathcal{N}\bigl(0,\lambda_V^{-1}\bigr),\\
& a_i \sim \mathcal{N}\bigl(0,\lambda_a^{-1}\bigr),\quad
b_i \sim \mathcal{N}\bigl(0,\lambda_b^{-1}\bigr),\quad
R_{i,j}|U,V \sim \mathcal{N}\bigl( \pscal{U_i}{V_j} + a_i + b_j, \tau^{-1}\bigr).
\end{align*}
In order to sample from the posterior distribution which we denote by $p(\Theta|R)$, where $\Theta = \{U,V,a,b,\lambda_U,\lambda_V,\lambda_a,\lambda_b\}$,
we use the following two-steps procedure:
\begin{enumerate}
	\item Sample from $p(U,V,a,b|R,\lambda_U,\lambda_V,\lambda_a,\lambda_b)$ using SGLD or SGLD-FP with a minibatch size of $5000$ observations with a step size $\gamma = 10^{-4}$. Sample for $1000$ steps before updating the weights $\lambda_U,\lambda_V,\lambda_a,\lambda_b$;
	\item Sample new $\lambda$ from $p(\lambda_U,\lambda_V,\lambda_a,\lambda_b|U,V,a,b)$ using the Gibbs sampler.
\end{enumerate}
The experiments are performed on the Movielens dataset $ml-100k$ (\href{https://grouplens.org/datasets/movielens/100k/}{link to dataset}).
We apply our control variates procedure as a postprocessing step following \cite{baker2019control}. The functional of interest is the mean squared error over the test subsample, $f(U,V,a,b) = \sum\nolimits_{(i,j) \in I_{\text{test}}}(R_{i,j}-\pscal{U_i}{V_j} - a_i - b_j)^2$. Since the dimension of parameter space is very high, first-order control variates are the only option among Stein's control variates. Parts of SGLD- and SGLD-FP-based trajectories before and after using control variates, and confidence intervals for estimation of $f$ are presented in \Cref{fig:bpmf_sgld}.
\begin{figure}[htbp]
\minipage{0.53\textwidth}
\begin{rline}
\includegraphics[width=0.492\linewidth]{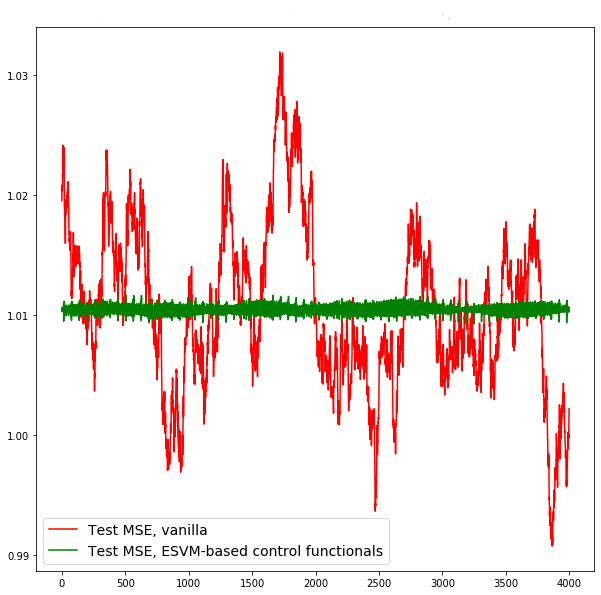}
\includegraphics[width=0.492\linewidth]{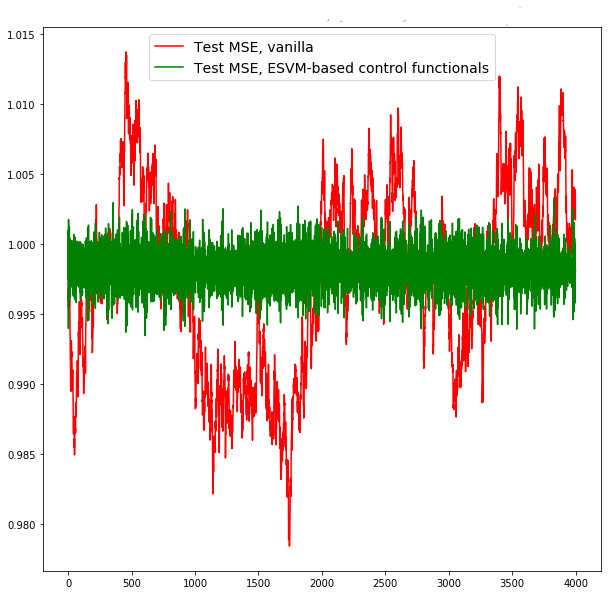}
\end{rline}
\endminipage\hfill
\minipage{0.47\textwidth}
\includegraphics[width=0.49\linewidth]{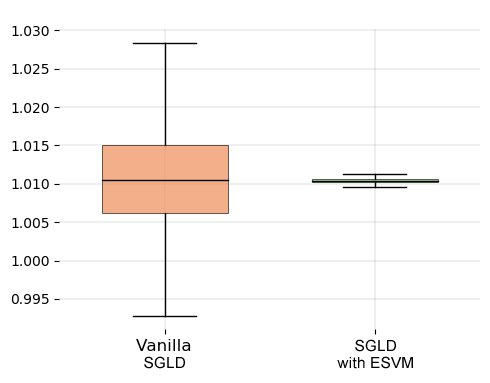}
\includegraphics[width=0.49\linewidth]{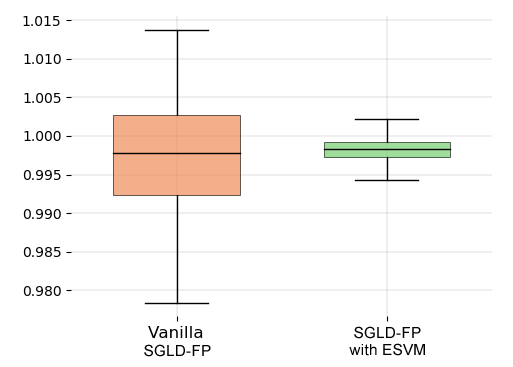}
\endminipage
\caption{\label{fig:bpmf_sgld}
Bayesian Probabilistic Matrix Factorization from \Cref{pp:bpmf}. Left Panel: test MSE trajectory for SGLD (left) and SGLD-FP (right) with and without ESVM. Right Panel: confidence intervals for test MSE trajectory for SGLD (left) and SGLD-FP (right).}
\end{figure}

\appendix
\section{Supplementary material for Variance reduction for dependent sequences with applications to  Stochastic Gradient MCMC}

\subsection{Proof of \Cref{asymptoticvariance}}\label{sec:proofasymptoticvariance}
With notation $\tih = h - \pi(h)$, we can represent the variance of $\pi_n(h)$, $h\in\H$, as
\begin{align*}
\E\Biggl[ \biggl(\frac{1}{n} \sum_{k=0}^{n-1} \tih(X_k) \biggr)^2\Biggr] =& \frac{1}{n^{2}} \Biggl\{ \sum_{k=0}^{n-1} \E \bigl[\tih^2(X_k)\bigr] + 2 \sum_{\ell=1}^{n-1} \sum_{k = 0}^{n-\ell-1} \E \bigl[ \tih(X_k) \tih(X_{k+\ell})\bigr]\Biggr\}.
\end{align*}
Multiplying the both sides by $n$ and subtracting $\rho^{(h)}(0) + 2 \sum_{\ell=1}^{n-1} (1 - \ell n^{-1})\rho^{(h)}(\ell)$, we get
\begin{align*}
&n\,\E\Bigg[ \bigg(\frac{1}{n} \sum_{k=0}^{n-1} \tih(X_k) \bigg)^2\Bigg] - \rho^{(h)}(0) - 2 \sum_{\ell=1}^{n-1} \bigg(1 - \frac{\ell}{n} \bigg)\rho^{(h)}(\ell)  \\
&\qquad\quad =\frac{1}{n} \sum_{k=0}^{n-1} \Bigl(\E \bigl[\tih^2(X_k)\bigr]  - \rho^{(h)}(0)\Bigr) +\frac{2}{n}\sum_{\ell=1}^{n-1} \sum_{k = 0}^{n-\ell-1} \Bigl(\E \bigl[\tih(X_k) \tih(X_{k+\ell})\bigr]- \rho^{(h)}(\ell)\Bigr).
\end{align*}
It follows from Cesaro mean theorem and \ref{CS} that the right-hand side tends to zero as $n \to \infty$. Similarly, 
$\rho^{(h)}(0) + 2 \sum_{\ell=0}^{n-1} (1 - \ell n^{-1})\rho^{(h)}(\ell) \to \sum_{\ell\in\zset}\rho^{(h)}(\ell)$ as $n \to \infty$.

\subsection{Proof of \Cref{th:main}}\label{sec:proofthmain}
Let us first start with a technical lemma
the proof of which we postpone to the end of the section. 
In what follows, set $\bVn(h):= \E [\Vn(h)]$. 
\begin{lemma}
\label{lem:techres}
    Let $\H$ be a class of functions with constant mean and
    assume that \ref{CS} and \ref{CD} hold.
    Then, for any $h \in \H$ and any $h_1,h_2\in\H$ 
    with $\| h_1 - h_2\|_{\ltwo(\pi)}\leq\eps$,
    \begin{align*}
        &1)\ \bigl|V_\infty(h) - \bVn(h)\bigr| \lesssim 
        \bigl(\CM  + \varsigma (1-\lambda)^{-1}\bigr)b_nn^{-1}
        +\varsigma (1-\lambda)^{-2}  n^{-1}
        +\varsigma  (1 - \lambda)^{-1} \lambda^{b_n/2},\\
        &2)\ \bigl|\bVn(h_1) - \bVn(h_2)\bigr| \lesssim 
        \sqrt{\CM\CH} b_nn^{-1/2}
        +\bigl(\CM + \varsigma (1 - \lambda)^{-1} \bigr)b_nn^{-1} \\
        &\qquad\qquad\qquad\qquad\qquad\qquad
        \qquad\qquad\qquad\qquad\qquad\qquad\quad
        +\bigl(\sqrt{\CM} n^{-1/2}+\sqrt{\CH}\bigr)b_n\eps.
    \end{align*}
\end{lemma}
Let $\sh$ be a function in $\H$ leading to the smallest $\bVn(h)$, 
that is,
\[
    \sh \in \argmin\nolimits_{h \in \mathcal{H}} \bVn(h).
\]
For simplicity, we assume that $\sh$ exists as all the following arguments can easily be adapted by considering an approximate minimizer.
We decompose the excess of the asymptotic variance as 
\begin{align}\label{eq:thmain1}
    &V_\infty(\hh[n,\eps]) - \inf\nolimits_{h\in\H}V_\infty(h)\nonumber\\
    &\qquad\ =
	V_\infty(\hh[n,\eps]) - \bVn(\hh[n,\eps])
	+ \bVn(\hh[n,\eps]) - \bVn(\sh)
	+ \bVn(\sh) - \inf\nolimits_{h\in\H}V_\infty(h)
	\nonumber\\
    &\qquad\ \leq
    2 \sup\nolimits_{h\in\H}\bigl|V_\infty(h) - \bVn(h)\bigr|
    +\bVn(\hh[n,\eps]) - \bVn(\sh).
\end{align}
To bound the first term in \eqref{eq:thmain1}, we apply \Cref{lem:techres} and obtain 
\begin{multline}\label{eq:thmain4}
    \sup\nolimits_{h\in\H}\bigl|V_{\infty}(h) - \bVn(h)\bigr|
    \\ \lesssim 
        \bigl(\CM  + \varsigma (1-\lambda)^{-1}\bigr)b_nn^{-1}
        +\varsigma (1-\lambda)^{-2}  n^{-1}
        +\varsigma  (1 - \lambda)^{-1} \lambda^{b_n/2}.
\end{multline}
It remains to bound the second term in \eqref{eq:thmain1}.
Let $\sh_\eps\in\H_\eps$ be any closest to $\sh$ point in $\ltwo(\pi)$-distance. By the definition of $\hh[n,\eps]$, $\Vn(\hh[n,\eps]) - \Vn(\sh_\eps)\leq0$. Hence,
\begin{align}\label{eq:thmain3}
	&\bVn(\hh[n,\eps]) - \bVn(\sh) \nonumber \\
	&\qquad\leq \bVn(\hh[n,\eps]) - \bVn(\sh) -  \bigl(\Vn(\hh[n,\eps]) - \Vn(\sh_\eps)\bigr) \nonumber\\
	&\qquad=\bVn(\hh[n,\eps]) - \bVn(\sh) - \bigl(\Vn(\hh[n,\eps]) - \Vn(\sh) \bigr) + \bigl(\Vn(\sh_\eps) - \Vn(\sh)\bigr) \nonumber \\
	&\qquad\leq {\sup\nolimits_{h\in\mathcal{H}_\eps} \bigl\{\bVn(h)  -  \Vn(h) \bigr\}} + {\bigl(\Vn(\sh) -\bVn(\sh)  \bigr)}
	+  {\bigl(\Vn(\sh_\eps) - \Vn(\sh)\bigr)}.
\end{align}
By assumption and the union bound, 
it holds for the first term in \eqref{eq:thmain3} that
\begin{equation*}
    \Pb \biggl(\, \sup_{h\in\mathcal{H}_\eps} 
    \bigl\{\bVn(h)  -  \Vn(h) \bigr\}>t\biggr)
	\leq
	|\mathcal{H}_\eps|	
	\sup_{h\in\mathcal{H}_\eps}
	\Pb \Bigl(\bVn(h)  -  \Vn(h)>t\Bigr)
	\leq |\mathcal{H}_\eps| \, \alpha_n(t).
\end{equation*}
The second term in \eqref{eq:thmain3} can be handled in the same way,
\begin{equation*}
\Pb \bigl(\Vn(\sh) -\bVn(\sh) >t \bigr) \leq \alpha_n(t).
\end{equation*}
The last term in \eqref{eq:thmain3} we represent as
\begin{align*}
	\Vn(\sh) - \Vn(\sh_\eps) 
	&=
	\Vn(\sh)- \Vn(\sh_\eps) - \bigl(\bVn(\sh)- \bVn(\sh_\eps)\bigr)  +\bigl(\bVn(\sh)- \bVn(\sh_\eps)\bigr).
\end{align*}
Now the union bound implies
\begin{align*}
	\Pb  \Bigl(\Vn(\sh) -\Vn(\sh_\eps)- \bigl(\bVn(\sh)- \bVn(\sh_\eps)\bigr) > t \Bigr) \leq 2\alpha_n(t).
\end{align*}
Furthermore, using \Cref{lem:techres} and the fact that $h^*_{\eps}$ is  $\eps$-close to $h^*$ in $\ltwo(\pi)$-distance,
\begin{multline*}
        \bigl|\bVn(\sh) - \bVn(\sh_\eps)\bigr| \\
        \lesssim 
        \sqrt{\CM\CH} b_nn^{-1/2}
        +\bigl(\CM + \varsigma (1 - \lambda)^{-1} \bigr)b_nn^{-1} 
        +\bigl(\sqrt{\CM} n^{-1/2}+\sqrt{\CH}\bigr)b_n\eps.
\end{multline*}
Combining these inequalities and substituting them into \eqref{eq:thmain3}, 
we obtain, with probability at least $1-(|\H_{\eps}|+3)\alpha_n(t)$, 
\begin{multline}\label{eq:thmain2}
\bVn(\hh[n,\eps]) - \bVn(\sh)
	\\\lesssim t
	+\sqrt{\CM\CH} b_nn^{-1/2}
	+\bigl(\CM + \varsigma (1 - \lambda)^{-1} \bigr)b_nn^{-1} 
    +\bigl(\sqrt{\CM} n^{-1/2}+\sqrt{\CH}\bigr)b_n\eps.
\end{multline}
Substituting \eqref{eq:thmain4} and \eqref{eq:thmain2} into 
\eqref{eq:thmain1} we conclude that, with the same probability, 
\begin{multline*}
    V_\infty(\hh[n,\eps]) - \inf\nolimits_{h\in\H}V_\infty(h)
    \lesssim
    t
    +\bigl(\sqrt{\CM}n^{-1/2}+\sqrt{\CH}\bigr)b_n \eps
	+\sqrt{\CM\CH}\,b_nn^{-1/2}\\
    +\bigl(\CM + \varsigma (1 - \lambda)^{-1}\bigr)b_nn^{-1}
    +\varsigma (1 - \lambda)^{-2}n^{-1}
    +\varsigma (1 - \lambda)^{-1}\lambda^{b_n/2},
\end{multline*}
where we have used the notation of \Cref{th:main}.
The proof 
is completed by taking $t = \alpha_n^{-1}\bigl(\delta/2|\H_\eps|\bigr)$
and assuming that $|\H_\eps|\ge3$ (this involves no loss of generality).
We are left with the task of proving \Cref{lem:techres}.
\\ \\
\noindent{\emph{Proof of \Cref{lem:techres}.}}
Let us first find a leading term in sample autocavariance function. 
Recall that for any $h\in\H$, $\tih = h-\pi(h)$. 
By expanding the brackets and 
adding/subtracting $\pi(h)$ in the definition~\eqref{eq:empirical-autorcovariance}, 
we get, for any $|\ell|\leq b_n$, 
\begin{align}\label{eq:eaf_decomposition}
    \rho^{(h)}_n(\ell) 
    = 
    A_{n,1}^{(h)}(\ell)
    + A_{n,2}^{(h)}(\ell)
    + A_{n,3}^{(h)}(\ell),
\end{align}
where, for $0\leq \ell\leq b_n$,
\begin{align*}
    A_{n,1}^{(h)}(\ell)=A_{n,1}^{(h)}(-\ell) 
    &:= 
    n^{-1}\sum\nolimits_{k=0}^{n-\ell-1} \tih(X_k) \tih(X_{k+\ell}),\\ 
    A_{n,2}^{(h)}(\ell) = A_{n,2}^{(h)}(-\ell) 
    &:= 
    - n^{-1}\pi_n(\tih) 
    \Bigl\{ \sum\nolimits_{k=0}^{n-\ell-1} \tih(X_k) + \sum\nolimits_{k=\ell}^{n-1} \tih(X_k) \Bigr\},\\
    A_{n,3}^{(h)}(\ell)= A_{n,3}^{(h)}(-\ell)
    &:= 
    (1-\ell/n) \, \pi_n^2(\tih).
\end{align*}
It follows that the leading term in this decomposition is $A_{n,1}^{(h)}(\ell)$. 
The remainder terms $A_{n,2}^{(h)}(\ell)$ and $A_{n,3}^{(h)}(\ell)$
can be bounded, under assumptions \ref{CS} and \ref{CD}, as follows.
\begin{align}\label{eq:bound-for-A3}
    \Bigl|\E\bigl[A_{n,3}^{(\tih)}(\ell)\bigr]\Bigr|
    &\leq n^{-2} \biggl\{
    \sum\nolimits_{k=0}^{n-1} \E\bigl[\tih^2(X_k)\bigr] + 2 \sum\nolimits_{\ell=0}^{n-1} \sum\nolimits_{k=1}^{n-\ell-1} \E\bigl[\tih(X_k) \tih(X_{k+\ell})\bigr] 
    \biggr\} \nonumber \\
    &\leq n^{-1}\biggl\{\rho^{(h)}(0) + 2 \sum\nolimits_{\ell = 0}^{n-1} (1 - \ell n^{-1}) \rho^{(h)}(\ell)\biggr\} + 3 \CM n^{-1} 
    \nonumber\\
    &\le \C n^{-1},
\end{align}
where \(\C := 2\varsigma (1 - \lambda)^{-1}+ 3\CM\). 
In the same manner we conclude that
\begin{align}\label{eq:bound-for-A2}
    \Bigl| \E\bigl[A_{n,2}^{(h)}(\ell)\bigr] \Bigr| 
    &\leq  
    n^{-1} \E^{1/2}\bigl[\pi_n^2(\tih)\bigr]
    \bigg\{ \E^{1/2}\biggl[ \Bigl( \sum\nolimits_{k=0}^{n-\ell-1} \tih(X_k)\Bigr)^2 
    + \Bigl( \sum\nolimits_{k=\ell}^{n-1} \tih(X_k) \Bigr)^2 
    \biggr]  \bigg\}\nonumber\\
    &\leq
    2 \C n^{-1}.
\end{align}
The last two bounds show that the last two terms in \eqref{eq:eaf_decomposition}
are of order $n^{-1}$.
Having disposed of this preliminary step, we can now return to 
statements of the lemma.
\\ \\
\noindent\emph{Statement 1.} 
From decomposition \eqref{eq:eaf_decomposition} and bounds \eqref{eq:bound-for-A3}, \eqref{eq:bound-for-A2}, we deduce that 
\begin{align*}
	\bigl|\bVn(h_1) - \bVn(h_2)\bigr| 
	&=
	\sum\nolimits_{|\ell|\leq b_n}w_n(\ell)\,
	\E\Bigl[ \rho_n^{(h_{1})}(\ell) - \rho_n^{(h_{2})}(\ell)
	\Bigr]\\
	&\leq 
	2 b_n 
	\max\nolimits_{|\ell|\leq b_n}
	\Bigl|
	\E\Bigl[A_{n,1}^{(h_1)}(\ell) - A_{n,1}^{(h_2)}(\ell)\Bigr]
	\Bigr|
	+ 12 \C b_n n^{-1}.
\end{align*}
With notation $\tih_{12} = \tih_1 - \tih_2$, 
it follows, for any $0\leq \ell \leq b_n$, that
\begin{align*}
	A_{n,1}^{(h_1)}(\ell) - A_{n,1}^{(h_2)}(\ell)
	=
	n^{-1}
	    \sum\nolimits_{k=0}^{n-\ell-1}
	    \Bigl(
	    \tih_1(X_k) \tih_{12}(X_{k+\ell})+ \tih_{12}(X_{k})\tih_2(X_{k+\ell})
	    \Bigr).
\end{align*}
Using Cauchy–Schwarz inequality (twice) and 
\ref{CS}, we have
\begin{align*}
	\E\biggl[\,
	    \sum_{k=0}^{n-\ell-1}
	    \tih_1(X_k) \tih_{12}(X_{k+\ell})
	\biggr]
	&\leq
	\E^{1/2}\biggl[\,
	    \sum_{k=0}^{n-\ell-1}
	    \tih_1^2(X_k) 
	\biggr]
	\E^{1/2}\biggl[\,
	    \sum_{k=0}^{n-\ell-1}
	    \tih_{12}^2(X_{k+\ell})
	\biggr]\\
	&\leq 
	\sqrt{\CM + (n-l)\rho^{(\tih_{1})}(0)}\,
	\sqrt{\CM + (n-l)\rho^{(\tih_{12})}(0)}.
\end{align*}
We now apply this argument again and 
obtain
\begin{align*}
	\E\Bigl[
	A_{n,1}^{(h_1)}(\ell) - A_{n,1}^{(h_2)}(\ell)
	\Bigr]
	\lesssim 
	\CM n^{-1} + \sqrt{\CM\CH}n^{-1/2}
	+ \bigl(\sqrt{\CM} n^{-1/2} +\sqrt{\CH}\bigr) 
	\bigl\| \tih_1 - \tih_2\bigr\|_{\ltwo(\pi)}.
\end{align*}
Finally, since $\| \tih_1 - \tih_2\|_{\ltwo(\pi)} \leq 2\| h_1 - h_2\|_{\ltwo(\pi)}$, we conclude
\begin{multline*}
	\bigl|\bVn(h_1) - \bVn(h_2)\bigr| 
	\\ \lesssim
	\bigl(\CM + \C \bigr) b_nn^{-1}
	+ \sqrt{\CM\CH} b_nn^{-1/2}
	+ \bigl(\sqrt{\CM}n^{-1/2}+\sqrt{\CH}\bigr)  b_n 
	\| h_1 - h_2\|_{\ltwo(\pi)}.
\end{multline*}
\noindent\emph{Statement 2.}
Let us denote
\[
    V_{n, \rho}(h) = \sum\nolimits_{|\ell|\leq b_n} w_n(\ell) \rho^{(h)}(\ell).
\]
With this notation, we have the following decomposition
\begin{equation}
\label{eq:decomposition}
    \bigl|V_{\infty}(h) - \bVn(h) \bigr| 
    \leq 
    \bigl|V_{\infty}(h) - V_{n, \rho}(h) \bigr|
    +\bigl|V_{n, \rho}(h) - \bVn(h)\bigr|.
\end{equation}
To bound the first term in the right-hand side of \eqref{eq:decomposition},
we represent it as
\begin{align*}
    |V_{n, \rho}(h) -  V_{\infty}(h)| \leq
    \sum\nolimits_{|\ell|\leq b_n}  |1 - w_n(\ell)| \, \bigl|\rho^{(h)}(\ell) \bigr| + \sum\nolimits_{|\ell| > b_n} \bigl|\rho^{(h)}(\ell)\bigr|.
\end{align*}
Using \ref{CD} and the fact that $w_n(\ell)=1$ for $\ell\in[-b_n/2,b_n/2]$, we obtain 
\[
    \sum\nolimits_{|\ell|\leq b_n} |1 - w_n(\ell)| \bigl|\rho^{(h)}(\ell) \bigr|
    = 2\sum\nolimits_{\ell = b_n/2}^{b_n} |1 - w_n(\ell)|\, \bigl|\rho^{(h)}(\ell)\bigr| \le 2 \varsigma  (1 - \lambda)^{-1} \lambda^{b_n/2}.
\]
In the same manner we can see that
\[
    \sum\nolimits_{|s| > b_n} \bigl|\rho^{(h)}(s)\bigr|  \le 2\varsigma  (1 - \lambda)^{-1} \lambda^{b_n}.
\]
Combining the last two bounds we conclude that
\begin{equation}
    |V_{n, \rho}(h) -  V_{\infty}(h)| \leq
    4\varsigma  (1 - \lambda)^{-1} \lambda^{b_n/2}.
    \label{eq:final-bound-2}
\end{equation}
Now let us turn to the second term in the right-hand side of \eqref{eq:decomposition}. 
The decomposition \eqref{eq:eaf_decomposition} and the bounds \eqref{eq:bound-for-A3}, \eqref{eq:bound-for-A2} yield
\begin{align*}
    |\bVn(h) -V_{n, \rho}(h)| 
    &= \sum\nolimits_{|\ell|\leq b_n} w_n(\ell) \Bigl(\E\bigl[\rho^{(h)}_n(\ell)\bigr]- \rho^{(h)}(\ell) \Bigr)
    \nonumber\\
    &\leq \sum\nolimits_{|\ell|\leq b_n}  
    \Bigl|\E\bigl[A_{n,1}^{(h)}(\ell)\bigr]- \rho^{(h)}(\ell) \Bigr|
    +6\C b_n n^{-1}.
\end{align*}
Using \ref{CS} and \ref{CD}, it follows that
\begin{align*}
    \Bigl|\E\bigl[A_{n,1}^{(h)}(\ell)\bigr] - \rho^{(h)}(\ell)\Bigr| 
    &\le 
    n^{-1}\sum\nolimits_{k=0}^{n-\ell-1} \Bigl| 
    \E\bigl[\tih(X_k) \tih(X_{k+\ell})\bigr] - \rho^{(h)}(\ell) \Bigr| + \ell n^{-1} \rho^{(h)}(\ell) \\
&\le \CM  n^{-1} + \varsigma \ell \lambda^\ell  n^{-1}.
\end{align*}
Combining these, we get
\begin{align}
    |\bVn(h) -V_{n, \rho}(h)| 
    &\leq \bigl(12 \varsigma (1-\lambda)^{-1}+ 20 \CM \bigr)b_n n^{-1}
    + 2\varsigma 
    (1-\lambda)^{-2}  n^{-1}.
    \label{eq:final-bound-1}
\end{align}    
Finally, we obtain the desired conclusion by
substituting \eqref{eq:final-bound-2} and \eqref{eq:final-bound-1} into \eqref{eq:decomposition}. 

\subsection{Proof of \Cref{prop:W2_concentration}} 
\label{sec:prop_W2}
1. The sequence $(\covcoeff(\ell))_{\ell=0}^{\infty}$ is symmetric and positive semidefinite by construction. By the Markov property, for any $k, \ell \in \nset$,
\begin{align*}
\E_{\startdist}\bigl[\tih(X_k) \tih(X_{k+\ell})\bigr]  - \covcoeff(\ell) =  \bar{\E}_\zeta \bigl[\tih(X) \phi_{\ell}(X) - \tih(X^{\prime}) \phi_{\ell}(X^{\prime})\bigr],
\end{align*}
where we denote \( \phi_{\ell}(x) := P^{\ell} \tih(x)\) and \(\zeta \in \Pi(\startdist P^k, \pi)\) is the optimal coupling of $\startdist P^k$ and $\pi$ in $W_2^{\metric}$-distance, $(X, X') \sim \zeta$.  Note that
\begin{multline*}
    \bigr| \bar \E_{\zeta} [\tih(X) \phi_{\ell}(X) - \tih(X')\phi_{\ell}(X')] \bigl| 
    \le \{\bar{\E}_{\zeta} [ \{\tih(X) - \tih(X')\}^2] \}^{1/2} \{\bar{\E}_{\zeta}[ \{\phi_{\ell}(X')\}^2] \}^{1/2} 
    \\
    +\{ \bar{\E}_{\zeta} [ \{\phi_{\ell}(X) - \phi_{\ell}(X')\}^2] \}^{1/2} \{\bar{\E}_{\zeta}[\{ \tih(X) \}^2]\}^{1/2}.
\end{multline*}
It is easy to check that \(\phi_{\ell}\) is a Lipschitz function,
\[
|\phi_{\ell}(x) - \phi_{\ell}(x')| \leq L W_2^{\metric}(\delta_x P^{\ell}, \delta_{x'} P^{\ell}) \leq L\wascoef_2^{\ell}\metric(x,x^{\prime})\eqsp.
\]
Since the Markov kernel $P$ is $W_2$-geometrically ergodic, we get
\begin{multline*}
\big| \bar{\E}_{\zeta} [\tih(X) \phi_{\ell}(X) - \tih(X')\phi_{\ell}(X^{\prime})] \big|
\\\le L W^{\metric}_2(\startdist P^k, \pi)  \{\bar{\E}_{\zeta}[ \{ \phi_{\ell}(X^{\prime}) \}^2]\}^{1/2} + L \wascoef_2^{\ell} W_2^{\metric}(\startdist P^k, \pi) \{\bar{\E}_{\zeta}[ \{ \tih(X)\}^2] \}^{1/2}.\end{multline*}
Let us compute  $\bar{\E}_{\zeta}[ \{\phi_{\ell}(X^{\prime})\}^2]= \pi( \phi_{\ell}^2)$. 
Since $P$ is $W_2^\metric$-geometrically ergodic, we have
$W_2^{\metric}(\delta_y P^\ell,\pi) \leq \wascoef_2^{\ell} W_2^{\metric}(\delta_y, \pi)$. 
Note also that $\pi(\tih)= 0$ implies $\pi (\phi_{\ell}) = 0$, hence
\begin{equation}
\label{eq:phi_2_moment}
\pi(\phi_{\ell}^2) = \int \bigl[\phi_{\ell}(y) - \int \phi_{\ell}(x)\pi(\rmd x)\bigr]^2 \pi(\rmd y) \leq L^2 \wascoef_{2}^{2\ell} \int \{W_2^{\metric}(\delta_{y},\pi)\}^2\pi(\rmd y) .
\end{equation}
Finally, we need to compute \(\bar{\E}_{\zeta}[\{ \tih(X)\}^2]= \startdist P^k\big(\tih^2\big)\). For an arbitrary $\hat{x} \in \Xset$,
\[
\bigl|\tilde{h}(x)\bigr|^2 = \biggl|\int \{h(x) - h(y)\}\,\pi(\rmd y) \biggr|^2 \leq 2L^2\bigl(\metric^2(x,\hat{x}) + \int \metric^{2}(x,\hat{x}) \pi(\rmd x)\bigr) .
\]
In order to bound $\startdist P^k(\metric^{2}(x,\hat{x}))$, we write
\begin{align*}
\int \metric^{2}(x,\hat{x}) \startdist P^k (\rmd x)  &= \iint  \metric^{2}(x,\hat{x}) \zeta(\rmd x\rmd x') \leq 2 \iint \metric^{2}(x,x^{\prime}) \zeta(\rmd x \rmd x') \\ 
&+2 \int \metric^{2}(x,\hat{x}) \pi(\rmd x)
\leq 2 \wascoef_2^{2k} \{W_2^{\metric}(\startdist,\pi)\}^2 + 2 \int \metric^{2}(x,\hat{x}) \pi(\rmd x)  \,. 
\end{align*}
Hence
\[
\startdist P^k \big(\tih^2\big) \leq 4 L^2 \int  \metric^{2}(x,\hat{x}) \pi(\rmd x) + 2 L^2 \wascoef_2^{2k} \{W_2^{\metric}(\startdist, \pi)\}^2,
\]
and
$\bigl|\E_{\startdist}[\tih(X_k) \tih(X_{k+\ell})]  - \rho^{(h)}(\ell)\bigr|  \leq A_{1}L^2 \wascoef_2^{k+\ell} W_2^{\metric}(\startdist, \pi)
$
where 
\begin{equation}
\label{eq:definition-C-1}
A_1 := 2\inf\nolimits_{\hat{x} \in \Xset} W^\metric_2(\delta_{\hat{x}},\pi) +  2 W_2^{\metric}(\startdist, \pi) + \biggl[\int \{W_2^{\metric}(\delta_{y},\pi)\}^2\pi(\rmd y)\biggr]^{1/2} \,.
\end{equation}
Summing the last inequality with respect to \(k\), we obtain
\begin{equation}
\label{eq:cov_decay}
\sum\nolimits_{k=0}^\infty \bigl|\E_{\startdist}[\tih(X_k) \tih(X_{k+\ell})]  - \rho^{(h)}(\ell)\bigr| \leq A_{1}L^2 (1 - \wascoef_2)^{-1} \wascoef_2^{\ell} W_2(\startdist, \pi)\eqsp.  
\end{equation}
Hence, the second assumption in~\ref{CS} holds  with $\CM$ defined in \eqref{eq:constatsW2_1}.
The third assertion clearly follows from \eqref{eq:cov_decay}. To check \ref{CD}, we write
\begin{align}
\label{eq:cd_w2}
|\covcoeff(\ell)| &= \biggl|\int \tilde{h}(x)\big[\delta_{x}P^{\ell}\big(h\big) - \pi(h)\big]\pi(\rmd x) \biggr| \leq L\int \big|\tilde{h}(x)\big|W_{2}^{\metric}(\delta_{x}P^{\ell},\pi) \pi(\rmd x) \nonumber \\
&\leq L\wascoef_2^{\ell} \int \big|\tilde{h}(x)\big|W_{2}^{\metric}(\delta_{x},\pi)\pi(\rmd x) \leq L\wascoef_2^{\ell} \sqrt{\CH} \biggl[\int \{W_2^{\metric}(\delta_{x},\pi)\}^2\pi(\rmd x)\biggr]^{1/2}\eqsp.
\end{align}
Hence \ref{CD} holds with $\lambda = \wascoef_2$ and 
$\varsigma =
     L\sqrt{\CH} \biggl[\int \{W_2^{\metric}(\delta_{x},\pi)\}^2\pi(\rmd x)\biggr]^{1/2}$.
\par
2. The proof essentially relies on \cite{Guillin2004}. Denote $Z_n(h) := \bigl(h(X_0),\ldots,h(X_{n-1})\bigr)$ and recall the representation \eqref{eq:sv}. It 
follows from \cite[Section~5.2]{bimns2020} that $\Vn(h)$ can be represented as a quadratic form 
\[
\Vn(h) = \pscal{A_n Z_n(h)}{Z_n(h)}\eqsp,
\]
where $A_n = n^{-1}(\Id - n^{-1}E) W (\Id - n^{-1} E)$,
$E$ is $n \times n$ matrix with elements $E_{j,k} = 1$ for any $1\leq j, k \leq n$, and $W$ is Toeplitz matrix with elements $W_{j,k} = w_n(j-k)$. Note that $\Vn(h)$ is invariant to shifts and, in particular, $\Vn(h) = \Vn(\tih)$. It is straightforward to show that $\|A_n\| \leq 2 b_n n^{-1}$, see \cite[Lemma~9]{bimns2020}. Furthermore, \cite[Corollary 18]{bimns2020} implies
\begin{equation}
\label{eq:concentr_bimns}
\Pb_{\startdist}\bigl(\bigl|\Vn(h) - \E_{\startdist}[\Vn(h)]\bigr| \geq t\bigr) \leq 2\exp{\biggl(-\frac{(1-\wascoef_2)^2t^2}{c\alpha L^2\big(\E_{\xi}[\|A_n Z_n(\tilde{h})\|^2] + t\|A_n\|\big)}\biggr)},
\end{equation}
where $c > 0$ is some universal constant.
By the Cauchy–Schwarz inequality, \newline
$\|A_n Z_n(h)\|^2 \leq \|A_n\|^2\|Z_n(h)\|^2$. Moreover, using \ref{CS}, we get
\begin{align}
\label{eq:squared_norm_bound}
\E_{\xi}[\|Z_n(\tilde{h})\|^2]\leq \CM + n\PVar[\pi][h] \leq \CM + n \sup\nolimits_{h \in \mathcal{H}}\PVar[\pi][h] = \CM + n\CH.
\end{align}
The statement follows from substitution \eqref{eq:squared_norm_bound} into \eqref{eq:concentr_bimns}. 

\subsection{Proof of \Cref{prop:W1_concentration}}
\label{sec:prop_W1} 
1. Proceeding similarly to \Cref{sec:prop_W2}, we use the Markov property to write, for $k, \ell \in \nset$,
\begin{multline*}
\bigl|\E_{\startdist}[\tih(X_k) \tih(X_{k+\ell})]  - \covcoeff(\ell)\bigr|  
\leq \{\bar{\E}_{\zeta} [ \{\tih(X) - \tih(X')\}^2] \}^{1/2} \{\bar{\E}_{\zeta}[ \{\phi_{\ell}(X')\}^2] \}^{1/2} \\
+\{ \bar{\E}_{\zeta} [ \{\phi_{\ell}(X) - \phi_{\ell}(X')\}^2] \}^{1/2} \{\bar{\E}_{\zeta}[\{ \tih(X) \}^2]\}^{1/2}\eqsp, 
\end{multline*}
where \(\zeta \in \Pi(\startdist P^k, \pi)\) is the optimal coupling of $\startdist P^k$ and $\pi$ in $W_1^{\metric}$ distance, $(X, X') \sim \zeta$. Since function $\tilde{h}$ is bounded and Lipschitz,
\begin{align*}
\bar{\E}_{\zeta} \bigl[ \{\tih(X) - \tih(X^{\prime})\}^2\bigr] \leq 2\B\bar{\E}_{\zeta}\big[\big|\tih(X) - \tih(X^{\prime})\big|\big] \leq 2\Lnorm  \B \wascoef_1^{k}W_{1}^{\metric}(\startdist,\pi). 
\end{align*}
Similarly, using that $\phi_{\ell}$ is bounded and Lipschitz (see \cref{sec:prop_W2} for the details),
\begin{align*}
\bar{\E}_{\zeta} \bigl[ \{\phi_{\ell}(X) - \phi_{\ell}(X^{\prime})\}^2\bigr] \leq 2\Lnorm \B\wascoef_1^{k+\ell}W_{1}^{\metric}(\startdist,\pi) .
\end{align*}
Proceeding as in \eqref{eq:phi_2_moment}, we obtain 
\begin{align*}
\bar{\E}_{\zeta}[ \{\phi_{\ell}(X')\}^2] \leq \iint \big[\phi_{\ell}(x) - \phi_{\ell}(y)\big]^{2}\pi(\rmd x)\pi(\rmd y) \leq 2\B \Lnorm\wascoef_1^{\ell}\int W_{1}^{\metric}(\delta_{y},\pi)\,\pi(\rmd y).
\end{align*}
Using the simple bound $\bar{\E}_{\zeta}[\{ \tih(X) \}^2] \leq 4\B^{2}$, it holds
\begin{equation}
\label{eq:cov_bound_W1}
\bigl|\E_{\startdist}[\tih(X_k) \tih(X_{k+\ell})]  - \covcoeff(\ell)\bigr| \leq 2\B C^{\prime}_{2}\wascoef_1^{(k+\ell)/2}  
\end{equation}
with 
\begin{equation}\label{eq:A2}
    C^{\prime}_2 = 2L\{W_1^{\metric}(\startdist,\pi)\}^{1/2}\biggl\{\int W_{1}^{\metric}(\delta_{x},\pi)\pi(\rmd x)\biggr\}^{1/2} + \bigl\{2 \Lnorm \B W_{1}^{\metric}(\startdist,\pi)\bigr\}^{1/2}.
\end{equation}
Hence, the second assumption in \ref{CS} holds with 
\begin{equation}
\label{eq:const_W1_M}
\CM = 2BC^{\prime}_{2}\bigl(1-\wascoef_1^{1/2}\bigr)^{-1},
\end{equation}
and the third one follows from \eqref{eq:cov_bound_W1}. Proceeding as in \eqref{eq:cd_w2},
\begin{align*}
\bigl|\covcoeff(\ell)\bigr| \leq 2L \B \wascoef_1^{\ell} \int W_{1}^{\metric}(\delta_{x},\pi)\pi(\rmd x)\eqsp, 
\end{align*}
and \ref{CD} holds with $\lambda = \wascoef_1$ and 
\begin{equation}
\label{eq:const_W1_varsigma}
\varsigma = 2 L\B \int W_{1}^{\metric}(\delta_{x},\pi)\pi(\rmd x) .
\end{equation}
2. Without loss of generality, we assume that $\| h \|_\infty \leq 1$. By Minkowski's inequality,
\begin{align*}
\bigl\|\Vn(h) - \E_{\startdist}\Vn(h)\bigr\|_{\xi,p} \leq \sum\nolimits_{\ell=-b_n+1}^{b_n}w_n(\ell)\bigl\|\ecovcoeff(\ell)-\E_{\startdist}[\ecovcoeff(\ell)]\bigr\|_{\xi,p},
\end{align*}
where $\|\cdot\|_{\startdist,p} := \bigl(\E_{\startdist}\bigl[\cdot\bigr]^{p}\bigr)^{1/p}$. For $\ell \in \nset_0$, we get
\begin{equation*}
\begin{split}
\ecovcoeff(\ell) &= \frac{1}{n}\sum\limits_{k=0}^{n-\ell-1}\bigl(h(X_k) - \pi_n(h)\bigr)\bigl(h(X_{k+\ell}) - \pi_n(h)\bigr) = \frac{1}{n}\sum\limits_{k=0}^{n-\ell-1}\tilde{h}(X_k)\tilde{h}(X_{k+\ell}) \\
&- \frac{1}{n}\bigl(\pi_n(h) - \pi(h)\bigr)^2 + \frac{\pi_n(h)-\pi(h)}{n}\sum\limits_{k=0}^{\ell-1}\bigl[\tilde{h}(X_k) + \tilde{h}(X_{n-\ell+k})\bigr] \\
&=: T_1 + T_2 + T_3.
\end{split}
\end{equation*}
Hence, 
\begin{align*}
\ecovcoeff(\ell) - \E_{\xi}[\ecovcoeff(\ell)] = \bigl(T_1 - \E_{\xi}[T_1]\bigr) + \bigl(T_2 - \E_{\xi}[T_2]\bigr) + \bigl(T_3 - \E_{\xi}[T_3]\bigr) =: \bar{T}_1 + \bar{T}_2 + \bar{T}_3 
\end{align*}
Now we proceed with estimating $\|\bar{T}_1\|_{\startdist,p}$. By \cite[Theorem~2]{doukhan1999} and \Cref{lem:covariance_est},
setting $\beta_{r,\ell} = 1 \wedge \wascoef_1^{r-\ell}$ and $A_1 = 4 \vee 256L\{ W_1(\xi, \pi) + 2\inf\nolimits_{\hat{x} \in \Xset}W_1(\delta_{\hat{x}},\pi) \}$,
\begin{equation}
\label{eq:t_1_1_bound}
    \|\bar{T}_1\|_{\startdist,p}^{p} \leq \frac{(2p-2)!}{(p-1)!}\rme^{p}\biggl[\biggl(n A_1\sum\limits_{r=0}^{n-1}\beta_{r,\ell}\biggr)^{p/2} \vee n A_1 16^{p-2}\sum\limits_{r=0}^{n-1}(r+1)^{p-2}\beta_{r,\ell}\biggr].
\end{equation}
It can be easily seen that
\begin{equation}
\label{eq:sum_bound}
\sum\limits_{r=0}^{n-1}(r+1)^{p-1}\beta_{r,\ell} \leq \frac{\ell^{p-1}}{p-1} + \frac{p!}{\wascoef_1^2}\biggl[\frac{1}{\log^{p}{(1/\wascoef_1)}} + \ell^{p}\biggr],\quad \sum\limits_{r=0}^{n-1}\beta_{r,\ell} \leq \frac{2\ell}{(1-\wascoef_1)}.
\end{equation}
Substituting \eqref{eq:sum_bound} into \eqref{eq:t_1_1_bound} and using Stirling's formula,
\begin{equation*}
\|\bar{T}_1\|_{\startdist,p}^{p} \leq 2^{2p}p^{p}\biggl[\biggl(\frac{2 A_1 \ell n}{1-\wascoef_1}\biggr)^{p/2}+n A_1 16^{p}\biggl\{\ell^{p} + \frac{2^{p}p^{p}p^{1/2}}{\rme^{p}\wascoef_1^{2}}\biggl(\frac{1}{\log^{p}{(1/\wascoef_1)}} + \ell^{p}\biggr)\biggr\}\biggr].
\end{equation*}
Since $\ell \leq b_n$, we obtain the following final bound on $\bar{T}_1$,
\begin{align}
\label{eq:T_1}
\|\bar{T}_{1}\|_{\xi,p} 
\leq 4p\biggl[\biggl(\frac{2b_n A_1 }{n(1-\wascoef_1)}\biggr)^{1/2} + 32b_nn^{1/p - 1} A_1^{1/p}\biggl(1 + \frac{p(1+\log{(1/\wascoef_1)})}{\rme\wascoef_1^{2/p}\log{(1/\wascoef_1)}}\biggr)\biggr].
\end{align}
Let us consider now $\bar{T}_2$ and $\bar{T}_3$. Using \ref{WE},
\begin{align*}
\bigl\Vert \pi_n(h) - \pi(h)\bigr\Vert_{\xi,p} \leq n^{-1}\Bigl\Vert \sum\nolimits_{k=0}^{n-1}h(X_k)-\E_{\xi}[h(X_k)] \Bigr\Vert_{\xi,p} + \frac{LW_1(\xi,\pi)}{n(1-\wascoef)}\eqsp.
\end{align*}
By \cref{lem:covariance_linear} and \cite[Theorem 2]{doukhan1999}, setting $A_2 = 4L \{ W_1(\xi, \pi) + 2\inf\nolimits_{\hat{x} \in \Xset}W_1(\delta_{\hat{x}},\pi) \}$,
\begin{equation*}
\Bigl\Vert\sum\nolimits_{k=0}^{n-1}h(X_k)-\E_{\xi}[h(X_k)]\Bigr\Vert_{\xi,p} \leq 4p\biggl[\frac{n^{1/2}A_2^{1/2}}{\sqrt{1-\wascoef_1}} \vee \frac{2pn^{1/p}A_2^{1/p}}{\wascoef_1^{1/p}\rme\log{(1/\wascoef_1)}}\biggr].
\end{equation*}
Now it holds for $\bar{T}_2$,
\begin{align}
\label{eq:T_2}
\|\bar{T}_{2}\|_{\xi,p} 
    &\leq 2 \bigl\Vert \pi_n(h) - \pi(h)\bigr\Vert^{2}_{\xi,2p} \leq \frac{4}{n^2}\Bigl\Vert\sum\limits_{k=0}^{n-1}h(X_k)-\E_{\xi}[h(X_k)]\Bigr\Vert_{\xi,2p}^2 + \frac{6L^2W^2_1(\xi,\pi)}{n^2(1-\wascoef)^2}  \nonumber\\
    & \leq 2^{6}p^2\biggl[\frac{A_2}{n(1-\wascoef_1)} \vee \frac{4p^2n^{2/p}A_2^{2/p}}{n^{2}\wascoef^{2/p}\rme^2\log^{2}{(1/\wascoef_1)}}\biggr] + \frac{4L^2W^2_1(\xi,\pi)}{n^2(1-\wascoef_1)^2} .
\end{align}
Finally, since $\ell \leq b_n$ and $h$ is bounded,
\begin{equation}
\label{eq:T_3}
\|\bar{T}_{3}\|_{\xi,p} \leq 16b_nn^{-1}.
\end{equation}
Using \eqref{eq:T_1}, \eqref{eq:T_2}, and \eqref{eq:T_3}, we get
\begin{equation}
\label{eq:p_norm}
\bigl\Vert \Vn(h) - \E_{\startdist}[\Vn(h)]\bigr\Vert_{\xi,p} \leq 2b_{n}\biggl[\frac{C^{\prime}_{1}pb_n^{1/2}}{n^{1/2}} + \frac{C^{\prime}_{2}p^2b_{n}}{n^{1-1/p}} + \frac{C^{\prime}_{3}p^4}{n^{2-2/p}}\biggr],
\end{equation}
where
\begin{equation}
\label{eq:rosenthal_constants}
\begin{split}
&C^{\prime}_1 = \frac{4(2A_1)^{1/2}}{\sqrt{1-\wascoef_1}},
\quad C^{\prime}_2 =  \frac{2^7A_1^{1/p}(1+2\log{(1/\wascoef_1)})}{\wascoef_1^{2/p}\log{(1/\wascoef_1)}} + \frac{2^{6}A_2}{1-\wascoef_1} + 12, \\
&C^{\prime}_3 = \frac{2^{8}A_2^{2/p}}{\rme^2\wascoef_1^{2/p}\log^2{(1/\wascoef_1)}} + \frac{A_2^{2}}{4(1-\wascoef_1)^2}.
\end{split}
\end{equation}
Under the assumption $p < n^{1/2}$, 
we obtain
\begin{equation}
\label{eq:fin_bound}
\bigl\Vert \Vn(h) - \E_{\startdist}[\Vn(h)]\bigr\Vert_{\xi,p} \leq b_{n}\|h\|_{\infty}^{2}\biggl[\frac{C_{R,1}pb_n^{1/2}}{n^{1/2}} + \frac{C_{R,2}p^2b_{n}}{n^{1-1/p}}\biggr].
\end{equation}
with $C_{R,1} = 2C_1^{\prime}$, $C_{R,2} = 2C_2^{\prime} +2C_3^{\prime}$. 
Now the statement follows from Markov's inequality.

\begin{lemma}
\label{lem:covariance_est}
Assume \ref{WE}-1. Let $h \in \Lip_{b,\metric}(\Lnorm, \B)$. For $i, m \in \nset_0$, we define  $\tilde{g}_{i,m}(x,x') = \tilde{h}(x) \tilde{h}(x') - c_{\xi,i,m}$, where  $c_{\xi,i,m}:=\E_{\xi}\left[g_{i,m}(X_i,X_{i+m})\right]$.
For $p,r,m \in \nset_{0}$, let
\begin{equation*}
C_{p,r,m}^{(h)} := \sup\Bigl\lvert\PCov_{\xi}\Bigl(\prod\nolimits_{k=1}^{u}\tilde{g}_{i_k,m}(X_{i_k},X_{i_k+m}),\,\prod\nolimits_{k=1}^{v}\tilde{g}_{j_k,m}(X_{j_k},X_{j_k+m})\Bigr)\Bigr\rvert,  
\end{equation*}
where the supremum is taken over all $0 \leq i_1 \leq \ldots \leq i_u < i_u + r \leq j_1 \leq \ldots \leq j_v \leq n$ with $u + v = p$. Then, for any $p, r \in \nset$,  
\begin{align*}
C^{(h)}_{p,r,m} \leq 
\begin{cases}
    2^{2p+2}\B^{2p}, & r \leq m , \\
    L\bigl\{ W_1(\xi, \pi) + 2\inf\nolimits_{\hat{x} \in \Xset}W_1(\delta_{\hat{x}},\pi) \bigr\} v2^{4p}B^{2p-1}\wascoef_1^{r-m}, & r > m .
\end{cases}
\end{align*}
\end{lemma}
\begin{proof}
Define the function 
\begin{equation*}
    G_{i_1,\ldots,i_u,m}(x_{i_1},x_{i_1+m},\ldots,x_{i_u},x_{i_u+m}) := \prod\nolimits_{k=1}^{u}\tilde{g}_{i_k,m}(x_{i_k},x_{i_k+m})\eqsp.
\end{equation*}
Let $
\mathsf{D}_{i_1,\ldots,i_u,m}= G_{i_1,\ldots,i_u,m} (X_{i_1},X_{i_1+m},\ldots,X_{i_u},X_{i_u+m})$. Since $\|G_{i_1,\ldots,i_u,m}\|_{\infty} \leq (2\B)^{2u}$ and $\|G_{j_1,\ldots,j_v,m}\|_{\infty} \leq (2\B)^{2v}$, we get
$C^{(h)}_{p,r,m} \leq 2^{2p+2}\B^{2p}.$
Now let $m < r$. Using Markov's property,
\begin{multline*}
\PCov_{\xi}\big(\mathsf{D}_{i_1,\ldots,i_u,m},\,\mathsf{D}_{j_1,\ldots,j_v,m}\big)  \\
= \E_{\xi}\bigl[\big(\mathsf{D}_{i_1,\ldots,i_u,m} - \E_{\xi}[\mathsf{D}_{i_1,\ldots,i_v,m}]\big)\big(P^{j_1-i_u-m}\varphi(X_{i_u+m}) - \pi(\varphi)\big)\bigr],
\end{multline*}
where
\[
\varphi(x) := \E_{x}\bigl[G_{j_1,\ldots,j_v,m}(x,X_{m},X_{j_2-j_1},X_{j_2-j_1+m},\ldots,X_{j_v-j_1},X_{j_v-j_1+m})\bigr].
\]
It follows from \Cref{lem:lip_const} that
$\lipnorm{\varphi} \leq \Lnorm v 2^{4v-1}B^{2v-1}$. 
By \cite[Theorem~20.1.2]{douc:moulines:priouret:2018}  $\lipnorm{P^{j_1-i_u-m}\varphi} \leq \wascoef_1^{j_1-i_u-m} \lipnorm{\varphi}$, and hence
\begin{align*}
\bigl\lvert P^{j_1-i_u-m}\varphi(x) - \pi(\varphi\big) \bigr\rvert \leq \Lnorm v 2^{4v-1}\B^{2v-1} \wascoef_1^{j_1-i_u-m}W_1(\delta_{x},\pi)\eqsp.
\end{align*}
This yields 
\begin{align*}
    \bigl\lvert\PCov_{\xi}\big(\mathsf{D}_{i_1,\ldots,i_u,m},\,\mathsf{D}_{j_1,\ldots,j_v,m}\big)\bigr\rvert  
    &\leq Lp2^{4p}\B^{2p-1}\wascoef_1^{j_1-i_u-m}\E_{\xi}\left[W_1(\delta_{X_{i_u+m}},\pi)\right].
\end{align*}
For a fixed $\hat{x} \in \Xset$, by the triangle inequality,
\[
W_1(\delta_x,\pi) \leq W_1(\delta_x,\delta_{\hat{x}}) + W_1(\delta_{\hat{x}},\pi) = \metric(x,\hat{x}) + W_1(\delta_{\hat{x}},\pi) .
\]
Since $\E_\xi[\metric(X_{i_u+m},\hat{x})] \leq W_1(\delta_\xi P^{i_u+m},\delta_{\hat{x}})$, we get
\begin{align*}
\E_{\xi}\left[\metric(X_{i_u+m},\hat{x})\right] \leq W_1(\xi P^{i_u+m}, \pi) + W_1(\delta_{\hat{x}},\pi) \leq \wascoef^{i_u+m}W_1(\xi, \pi) + W_1(\delta_{\hat{x}},\pi) ,
\end{align*}
showing that
$
\E_{\xi}\left[W_1(\delta_{X_{i_u+m}},\pi)\right] \leq W_1(\xi, \pi) + 2W_1(\delta_{\hat{x}},\pi)$.  The proof is complete.
\end{proof}

\begin{lemma}
\label{lem:covariance_linear}
Assume \ref{WE}-1. Let $h \in \Lip_{b, \metric}(\Lnorm, \B)$. For $p,r \in \nset_{0}$, we define
\begin{equation*}
C_{p,r} := \sup\biggl\lvert\PCov_{\xi}\biggl(\prod_{k=1}^{u}\big(h(X_{i_k}) - \E_{\xi}[h(X_{i_k})]\big),\,\prod_{k=1}^{v}\big(h(X_{j_k}) - \E_{\xi}[h(X_{j_k})]\big)\biggr)\biggr\rvert ,    \end{equation*}
where the supremum is taken over all $0 \leq i_1 \leq \ldots \leq i_u < i_u + r \leq j_1 \leq \ldots \leq j_v \leq n$ with $u + v = p$. Then for any $p, r \in \nset$,  
\begin{align*}
C_{p,r} \leq \Lnorm \bigl\{ W_1(\xi, \pi) + 2\inf\nolimits_{\hat{x} \in \Xset}W_1(\delta_{\hat{X}},\pi) \bigr\} p2^{2p}\B^{2p-1}\wascoef_{1}^{r} .
\end{align*}
\end{lemma}
\begin{proof} The proof is along the same lines as \Cref{lem:covariance_est} and is omitted.
\end{proof}

\begin{lemma}
\label{lem:lip_const}
Assume \ref{WE}-1. Set
\[
    \varphi(x) = \E_{x}\bigl[G_{j_1,\ldots,j_v,m}(x,X_{m},X_{j_2-j_1},X_{j_2-j_1+m},\ldots,X_{j_v-j_1},X_{j_v-j_1+m})\bigr],
\]
where $G_{j_1,\ldots,j_v,m}$ defined in \Cref{lem:covariance_est}. Then
\begin{equation*}
\lipnorm{\varphi} \leq \Lnorm v2^{4v-1}\B^{2v-1}.
\end{equation*}
\end{lemma}
\begin{proof}
We split the proof into two parts. First, we estimate Lipschitz constant of $g(x) =  \E_{x}\bigl[\prod_{k=1}^{v}\tilde{h}(X_{i_k})\tilde{h}(X_{i_{k+m}})\bigr]$ for $0 = i_1 \leq i_2 \leq \ldots \leq i_v$ and $m > 0$. Note that
$g(x) = \tilde{h}^{n_1}(x)\E_{x}\bigl[\prod_{k=2}^{b}\tilde{h}^{n_k}(X_{m_k})\bigr]$ where $0 = m_1 < m_2 < \ldots < m_b$ are distinct indices among $(i_1,i_1+m,\ldots,i_v,i_v+m)$ and $(n_1,\ldots,n_b)$ are their associated multiplicities ($\sum\nolimits_{k=1}^{b}n_k = 2v$). Hence, applying \Cref{lem:Lipschitz} with $f_{i} = \tilde{h}$  and $K = 2\B$, we get
$\lipnorm{g} \leq 2 \Lnorm v(2\B)^{2v-1}$.
Now we estimate Lipschitz constant of 
\begin{equation}
\label{eq:varphi_lip}
\varphi(x) = \E_{x}\Bigl[\prod\nolimits_{k=1}^{v}\bigl(\tilde{h}(X_{j_k-j_1})\tilde{h}(X_{j_k-j_1+m}) - c_{\xi,j_k,m}\bigr)\Bigr],
\end{equation}
where  $c_{\xi,j_k,m} := \E_{\xi}\bigl[\tilde{h}(X_{j_k})\tilde{h}(X_{j_k+m})\bigr]$. Expanding \eqref{eq:varphi_lip}, we obtain
\begin{equation}
\label{eq:phi_lip}
\varphi(x) = \sum\nolimits_{(\delta_1,\ldots,\delta_v)}(-1)^{\sum_{k}\delta_k}\E_x\left[\prod\nolimits_{k=1}^{v}\tilde{h}^{\delta_k}(X_{j_k})\tilde{h}^{\delta_k}(X_{j_k+m})\right] \, c_{\xi,j_k,m}^{1-\delta_k},
\end{equation}
where sum is taken w.r.t all $(\delta_1,\ldots,\delta_v)$ with $\delta_i \in \{0,1\}$. Note that all terms in the decomposition \eqref{eq:phi_lip} are Lipschitz. Since $|c_{\xi,j_k,m}| \leq 4\B^2$, we get
\begin{align*}
\lipnorm{\varphi} &\leq \sum\limits_{s=1}^{v} 2Ls\binom{v}{s}(2\B)^{2s-1}(2\B)^{2v-2s} = 2 \Lnorm v(2\B)^{2v-1}\sum\limits_{s=1}^{v}\frac{(v-1)!}{(s-1)!(v-s)!} \\
&= \Lnorm v2^{4v-1}\B^{2v-1}.
\end{align*}
\end{proof}

\begin{lemma}
\label{lem:Lipschitz}
Assume \ref{WE}-1. For any $b,v \geq 1$, $0 = i_1 < \dots < i_b \leq n $ and $n_k \in \nset, \sum\nolimits_{k=1}^{b}n_k = v$ we define $
g(x) = \E_{x}\bigl[\prod\nolimits_{k=1}^{b}f^{n_k}_{i_k}(X_{i_k}) \bigr]$,
where  $f_{i_k} \in \Lip_{b,\metric}(\Lnorm, K)$. Then  $\lipnorm{g} \leq \Lnorm v K^{v-1}$.
\end{lemma}
\begin{proof} Note that $
g(x) = f_{0}^{n_1}(x)\E_{x}\bigl[\prod_{k=2}^{b}f_{i_k}^{n_k}(X_{i_k})\bigr]$.
We proceed by induction in the number of distinct indices $b$. If $b=1$, then, for any $v \in \nset$,
\begin{align*}
|f_{0}^{v}(x) - f_{0}^{v}(y)| = |f_{0}(x) - f_{0}(y)|\cdot\Bigl\lvert \sum\nolimits_{k=0}^{v-1}f_{0}^{k}(x)f_{0}^{v-k-1}(y)\Bigr\rvert \leq vLK^{v-1}\metric(x,y)\eqsp.
\end{align*}
Assume $b>1$. Since $g(x)= f_0^{n_1}(x) P^{i_2} g_1(x)$ with $g_1(x) = \E_{x}\left[\prod_{k=2}^{b}f_{i_k}^{n_k}(X_{i_k - i_2})\right]$,
\begin{multline*}
|g(x) - g(y)| \leq \bigl\lvert f_{0}^{n_1}(x) - f_{0}^{n_1}(y) \bigr\rvert \bigl\lvert P^{i_2} g_1(x)\bigr\rvert
+ \bigl\lvert f_{0}^{n_1}(x) \bigr\rvert  \bigl\lvert P^{i_2} g_1(x) - P^{i_2} g(y)\bigr\rvert.
\end{multline*}
The function $g_1$ depends on $b-1$ indices and $\sum\nolimits_{k=2}^{b}n_k = v-n_1$. The induction assumption and \cite[Theorem~20.1.2]{douc:moulines:priouret:2018} show under \ref{WE}-1 that 
$\lipnorm{P_1^{i_2} g_1} \leq \lipnorm{g_1} \leq \Lnorm (v-n_1)K^{v-n_1-1}$. Observe that
\begin{equation*}
\lipnorm{g} \leq n_1 \Lnorm K^{n_1-1} K^{v-n_1}  + \Lnorm (v-n_1)K^{v-n_1-1} K^{n_1} \eqsp,
\end{equation*}
and the proof is complete.
\end{proof}

\subsection{Proof of \Cref{lem:SGLDW2contraction}\label{sec:SGLD_contraction}}
We provide the proof only for SGLD, since its adaptation to SGLD-FP is straightforward. Let $x = (\theta^{(1)}_{0},\tilde{S}^{(1)}_{0}),\, y = (\theta^{(2)}_{0},\tilde{S}^{(2)}_{0})$. We use the standard synchronous coupling technique adapted from \cite[Lemma 1]{brosse2018promises}. Let $(\xi_k)_{k \geq 0}$ be a sequence of i.i.d. $d$-dimensional Gaussian random variables, $(S_k)_{k \geq 0}$ and $(\tilde{S}_k)_{k \geq 0}$ be independent mini-batches with $|S_k| = |\tilde{S}_k| = M$.  
Set $(\theta_0^{(1)}, \theta_0^{(2)})= (x,y)$ and define recursively for $k \geq 0$,
\begin{align*}
&\theta_{k}^{(i)} =\theta_{k-1}^{(i)} - \gamma G(\theta_{k-1}^{(i)},S_{k}) + \sqrt{2\gamma}\xi_k; \\
&G(\theta,S) = \nabla U_0(\theta)  + N M^{-1} \sum\nolimits_{i \in S}\nabla U_i(\theta).
\end{align*}
Finally, define the sequences $(X_n^{(i)})_{n \geq 0}$, $i=1,2$ 
as $X_n^{(i)} = (\theta^{(i)}_{n},\tilde{S}_{n})$, for any $n \geq 0$.  Since $X_k^{(1)}$ and $X_k^{(2)}$ are distributed according to $\delta_{x} \Xkernel^{k}$ and $\delta_{y} \Xkernel^{k}$ respectively,
\begin{align*}
W^2_2(\delta_{x} \Xkernel^{k}, \delta_{y} \Xkernel^{k}) \leq \E\bigl[\metric^{2}(X_k^{(1)},X_k^{(2)})\bigr] = \E\bigl[\|\theta_{k}^{(1)}  - \theta_{k}^{(2)}\|^2\bigr].
\end{align*}
The rest of the proof follows \cite[Lemma 1]{brosse2018promises} and is omitted.

\clearpage
\bibliographystyle{siamplain}
\bibliography{bibliographie,evm}

\end{document}